\documentclass[12pt]{amsart}
\usepackage{amssymb}
\newcommand{\SBim}{\operatorname{ASBim}}
\newcommand{\HC}{\operatorname{HC}}
\newcommand{\Hilt}{\operatorname{HC-tilt}}
\newcommand{\h}{\mathfrak{h}}
\newcommand{\U}{\mathcal{U}}
\newcommand{\F}{\mathbb{F}}
\newcommand{\Hom}{\operatorname{Hom}}

\newcommand{\St}{\mathsf{St}}
\newcommand{\Coh}{\operatorname{Coh}}
\newcommand{\Dcal}{\mathcal{D}}
\newcommand{\Tcal}{\mathcal{T}}
\newcommand{\Ext}{\operatorname{Ext}}
\newcommand{\End}{\operatorname{End}}
\newcommand{\B}{\mathcal{B}}
\newcommand{\Ecal}{\mathcal{E}}
\newcommand{\Spec}{\operatorname{Spec}}
\newcommand{\g}{\mathfrak{g}}
\newcommand{\Acal}{\mathcal{A}}
\newcommand{\Ocal}{\mathcal{O}}
\newcommand{\Z}{\mathbb{Z}}
\newcommand{\Nilp}{\mathcal{N}}
\newcommand{\Dual}{\mathbb{D}}
\newcommand{\Ocat}{\mathsf{O}}
\newcommand{\Perv}{\operatorname{Perv}}
\newcommand{\Asf}{\mathsf{A}}

\newcommand{\pr}{\operatorname{pr}}
\newcommand{\Ring}{\mathsf{R}}
\newcommand{\Str}{\mathcal{O}}
\newcommand{\Br}{\mathsf{Br}}
\newtheorem{Thm}{Theorem}[section]
\newtheorem{Prop}[Thm]{Proposition}
\newtheorem{Cor}[Thm]{Corollary}
\newtheorem{Lem}[Thm]{Lemma}
\theoremstyle{definition}
\newtheorem{Ex}[Thm]{Example}
\newtheorem{defi}[Thm]{Definition}
\newtheorem{Rem}[Thm]{Remark}

\numberwithin{equation}{section}
\title{On modular Soergel bimodules, Harish-Chandra bimodules, and category O}
\author{Ivan Losev}
\address{Department
of Mathematics, Yale University, New Haven CT USA}
\email{ivan.loseu@gmail.com}
\begin{document}
\begin{abstract}
In this paper we continue the study of the category of modular Harish-Chandra bimodules initiated by
Bezrukavnikov and Riche and also study the modular version of the BGG category $\mathcal{O}$.
We prove a version of the Bezrukavnikov-Mirkovic-Rumynin localization theorem for the Harish-Chandra bimodules and for the category $\mathcal{O}$. We also relate the category of Harish-Chandra bimodules
to the affine Hecke category building on the prior work of Bezrukavnikov and Riche.
\end{abstract}
\maketitle
\tableofcontents
\section{Introduction}\label{S_intro}
Let $\F$ be an algebraically closed field of characteristic $p$.
Let $G$ be a simple algebraic group over $\F$.
 Throughout the paper  we assume that $p$ is bigger than $h$, the Coxeter number
of $G$.

 Let $W$ denote the Weyl group of $G$
and let $T$ be a maximal torus. We write $\g$ for the Lie algebra of $G$
and $\h$ for the Lie algebra of $T$. We write $\Lambda_r$ for the root lattice
and $\Lambda$ for the character lattice of $T$ so that $\Lambda_r\subset \Lambda$.
Consider the affine Weyl group $W^a:=W\ltimes \Lambda_r$ as well
as the extended affine group $W^{ea}:=W\ltimes \Lambda$. Recall that
$W^{ea}=(\Lambda/\Lambda_r)\ltimes W^a$, where $\Lambda/\Lambda_r$ is the subgroup of all
elements of length $0$. Below we will always view $\Lambda/\Lambda_r$ as a subgroup of
$W^{ea}$ in this way. For $\lambda\in \Lambda$ we write $t_\lambda$ for the corresponding element of
$W^{ea}$ (or for its natural lift to the corresponding braid group $\operatorname{Br}^{ea}$).

The goal of this paper is to relate several categories (triangulated, additive and abelian)
that are associated to the Lie algebra $\g$. Indecomposable (in the additive setting)
and simple (in the abelian setting) objects in these categories are labelled by the
elements of $W^{ea}$. The categories of interest are the affine Hecke category,
the category of Harish-Chandra bimodules and the (classical modular) category $\Ocat$.

\subsection{Affine Hecke category}\label{SS_Hecke_cat}
We start by discussing the affine Hecke category for $W^{ea}$.
The group $W^{ea}$ acts on $\h^*$
via its projection to $W$. To this realization one can assign the category
of (diagrammatic) Soergel bimodules  in the sense of Elias and Williamson,
\cite{EW}. Note that this paper deals with the case of a Coxeter system, the extension
that we need to deal with $W^{ea}$ can be found in \cite[Section 3]{Elias}.
The resulting category is a Karoubian monoidal category.

There is another, ``algebraic'' realization of the affine Hecke category due to
Abe, \cite{Abe} (for Coxeter groups, a modification for $W^{ea}$ can be found in
\cite[Section 2.2]{BR}). Abe's construction gives
a version of the classical construction of Soergel, see, e.g., \cite{Soergel}. Unlike Soergel's original construction, Abe's works well in positive characteristic.
We will use the version of \cite{BR}. Denote Abe's category for $W^{ea}$
by $\SBim$ (``A'' for Abe, ``S'' for Soergel and ``Bim'' for bimodules).
The Hom spaces in $\SBim$ are graded $\F[\h^*]$-bimodules supported on
the graph of the $W$-action on $\h^*$ and are finitely generated free left $\F[\h^*]$-modules and
also finitely generated free right $\F[\h^*]$-modules. It therefore makes sense to consider the completed
version $\SBim^\wedge$ of $\SBim$. In this category we have the same objects.
Let $\Ring$ be the completion of $\F[\h^*]$ at $0$.
For two objects
$B,B'\in \SBim^\wedge$ we set
$$\Hom_{\SBim^\wedge}(B,B')=\Hom_{\SBim}(B,B')\otimes_{\F[\h^*]}\Ring.$$
Note that the right hand side coincides with the completion on the left as well.
So $\SBim^\wedge$ is still a  monoidal category. One can show that there is a natural bijection between
the indecomposable objects (up to grading shift) in $\SBim$ and the indecomposable objects in
$\SBim^\wedge$. Both are labelled by the elements of $W^{ea}$. The category $\SBim^\wedge$
is generated by the Bott-Samelson objects $B^{AS}_s$, where $s$ runs over the set of simple
affine reflections, and the standard objects $\Delta^{AS}_x$ for $x\in \Lambda/\Lambda_r$.
These objects will be recalled in Section \ref{SS_BR_AS}.

This is an additive version of the affine Hecke category. One also could (and should) consider
the triangulated version, $K^b(\SBim^\wedge)$. Moreover, there are abelian versions. For example,
there is a highest weight category $\Ocat_\Ring$ (to be referred to as a ``Soergel-type''
category $\Ocat$) that has essentially appeared in \cite[Section 6]{EL}. The category
$\SBim^\wedge$ is identified with the category of tilting objects in $\Ocat_\Ring$.

One can expect two kinds of geometric realizations of $\SBim$. There is a constructible realization,
see \cite[Theorems 1.3,1.5]{BR2}.

\subsection{Harish-Chandra bimodules}
In this paper we will care
about a coherent realization of $\SBim^\wedge$. As suggested in \cite{BR}, to get this realization one uses a modular version of the category of Harish-Chandra bimodules (a classical object of study in the ``usual'', i.e., characteristic $0$, Lie representation theory).

Let $\U:=U(\g)$ be the universal enveloping algebra of $\g$.
Consider the completed version $\U^{\wedge_0}$ of the universal enveloping
algebra as in \cite[Section 3.5]{BR} at the zero Harish-Chandra (shortly, HC) character. It makes sense to speak about HC bimodules
for $\U^{\wedge_0}$,  see \cite[Section 3.5]{BR}. These are finitely generated
$G$-equivariant $\U^{\wedge_0}$-modules, where the resulting left action of $\U$
factors through $\U^{\wedge_0}$. Denote this category by $\HC^G(\U^{\wedge_0})$.
Inside we consider the category of ``HC-tilting'' objects, the direct
$G$-equivariant right $\U^{\wedge_0}$-module summands
in objects of the form $T\otimes_{\F}\U^{\wedge_0}$, where $T$ is a tilting $G$-module.
Denote the full subcategory of HC-tilting objects by $\Hilt^G(\U^{\wedge_0})$. Note that the category
$\HC^G(\U^{\wedge_0})$ is monoidal and $\Hilt^G(\U^{\wedge_0})$ is a monoidal Karoubian additive subcategory. Let us give examples of objects in $\Hilt^G(\U^{\wedge_0})$ considered in
\cite{BR}. There are reflection bimodules, $B^{HC}_s$,
labelled by the simple affine reflections $s$ (tensoring with such a bimodule gives a
classical reflection functor, hence the name) and standard bimodules $\Delta^{HC}_x,x\in \Lambda/\Lambda_r$ (tensoring with them gives a translation equivalence). We will elaborate on them in Section \ref{SS_BR_HC}.

The main result of \cite{BR}, mainly \cite[Theorem 6.3]{BR}, can be stated as follows
(we will elaborate why in Section \ref{SS_BR_HC}, see Proposition \ref{Prop:BR}).

\begin{Thm}\label{Thm:BR}
There is a full embedding $\SBim^\wedge\hookrightarrow \HC^G(\U^{\wedge_0})$ of monoidal
categories. This embedding sends $B^{AS}_s$ to $B^{HC}_s$ for all simple affine
reflections $s$ and $\Delta^{AS}_{x}$ to $\Delta^{HC}_x$ for all $x\in \Lambda/\Lambda_r$.
\end{Thm}

One can easily see that the image of this embedding lies in $\Hilt^G(\U^{\wedge_0})$.
One of the goals of this paper is to prove the following stronger version of this theorem.

\begin{Thm}\label{Thm:additive}
The full embedding $\SBim^\wedge\hookrightarrow \Hilt^G(\U^{\wedge_0})$ from Theorem \ref{Thm:BR}
is a category equivalence.
\end{Thm}

Here is a derived version of Theorem \ref{Thm:additive}.

\begin{Thm}\label{Thm:derived}
The equivalence $\SBim^\wedge\xrightarrow{\sim} \Hilt^G(\U^{\wedge_0})$ from Theorem \ref{Thm:additive}
extends to an exact monoidal equivalence $K^b(\SBim^{\wedge})\xrightarrow{\sim}
D^b(\HC^G(\U^{\wedge_0}))$ of triangulated categories.
\end{Thm}

We will also have the following result that actually plays an important role in
proving Theorems \ref{Thm:additive},\ref{Thm:derived}. This result should be
thought of as a localization theorem (a.l.a. \cite{BMR}) for Harish-Chandra
bimodules.

\begin{Thm}\label{Thm:localization} We have a monoidal exact equivalence of
triangulated categories $$D^b(\HC^G(\U^{\wedge_0}))\xrightarrow{\sim}
D^b(\Coh^{G^{(1)}}(\St^{(1)}_\Ring)).$$
\end{Thm}
Here $G^{(1)}$ denotes the Frobenius twist of $G$ and $\St^{(1)}_\Ring$
is a version of the Steinberg variety formally defined  in
Section \ref{SS_der_loc}.

Theorems \ref{Thm:derived}, \ref{Thm:localization} combined together give
an equivalence
$$K^b(\SBim^\wedge)\xrightarrow{\sim} D^b(\Coh^{G^{(1)}}(\St^{(1)}_\Ring)).$$
This is a coherent realization of the triangulated version of the Hecke category.


The equivalence from Theorem \ref{Thm:derived} is t-exact with respect to perverse t-structures,
compare to \cite[Theorems 54,55]{B}.
To avoid technicalities we will only establish this for
specialized categories. This is actually the result that is used to
prove Theorems \ref{Thm:additive} and Theorem \ref{Thm:derived}. Let $\Ocat$
denote the specialization to $\F$ of the highest weight category $\Ocat_\Ring$ mentioned
in Section \ref{SS_Hecke_cat}. 
On the other hand, one can consider the central reduction $\U_0$ of $\U$
and the category  of HC $\U$-$\U_0$-bimodules with trivial central character on the left.
This category will be denoted by $\HC^G(\U)_0$. The derived category $D^b(\HC^G(\U)_0)$
comes  with a so called perverse t-structure, to be recalled in   Section
\ref{SS_perverse}. Let $\Perv(\HC^G(\U)_0)$ denote the heart of this t-structure.

\begin{Thm}\label{Thm:abelian}
The full embedding from Theorem \ref{Thm:BR} gives rise to an
equivalence of abelian categories $\Ocat\xrightarrow{\sim} \Perv(\HC^G(\U)_0)$
and an equivalence $D^b(\Ocat)\xrightarrow{\sim} D^b(\HC^G(\U)_0)$
of triangulated categories.
\end{Thm}

We note that many results mentioned above were also independently obtained by Bezrukavnikov
and Riche, \cite{BR_new}.

\subsection{Category $\Ocat^{cl}$}
Another category we consider is the modular version of the classical BGG category $\Ocat$
to be denoted by $\Ocat^{cl}$. By definition, this is a category of finitely generated
strongly $B$-equivariant $\U$-modules, where, as usual, $B$ denotes a Borel subgroup
of $G$. The category $\Ocat^{cl}$ splits into blocks, let $\Ocat^{[0]}$ denote the principal one. We will see below that the simple objects in $\Ocat^{[0]}$ are indexed by the elements of $W^{ea}$. An important difference
of $\Ocat^{[0]}$ from the categories $\Ocat$ and $\HC^G(\U)_0$ is that it is ``periodic'':
the twist with any character of the Frobenius twist $B^{(1)}$ gives a self-equivalence
of $\Ocat^{[0]}$. We can think of the character lattice of $B^{(1)}$ as the lattice
$\Lambda\subset W^{ea}$. On the level of labels of simple objects, the equivalence
corresponding to $\lambda\in \Lambda$ acts on $W^{ea}$ by the right shift by $t_\lambda$.

Here is a basic (and easy) result relating the category $\Ocat^{[0]}$ to the categories
mentioned in the previous sections.

\begin{Prop}\label{Prop:O_HC}
We have a derived equivalence $D^b(\HC^G(\U)_0)\xrightarrow{\sim} D^b(\Ocat^{[0]})$.
\end{Prop}

Thanks to this proposition and Theorem \ref{Thm:abelian}, we get a derived equivalence
\begin{equation}\label{eq:derived_equiv_Os}
D^b(\Ocat^{[0]})\xrightarrow{\sim}D^b(\Ocat).
\end{equation}

%


\subsection{Noncommutative Springer resolution}
An important ingredient in the proof of several results, which is also of independent interest,
is the Noncommutative Springer resolution, \cite{BM}. In our context, this is
an $\F[\g^{*(1)}]\otimes_{\F[\g^{*(1)}]^{G^{(1)}}}\Ring$-algebra $\Acal_\Ring$ that serves as a noncommutative
resolution for various objects associated to the nilpotent cone of $\g^{(1)}$. For example,
we can consider the specialization
$$\Acal:=\Acal_\Ring\otimes_{\Ring}\F,$$
this is an algebra over
$\F[\Nilp^{(1)}]$, where $\Nilp^{(1)}$ is the nilpotent cone in $\g^{*(1)}$. Let $\tilde{\Nilp}^{(1)}$
be the Springer resolution of $\Nilp^{(1)}$. Then $\Acal$ is a noncommutative resolution of
$\Nilp^{(1)}$ meaning that there is an $\F[\Nilp^{(1)}]$-linear derived equivalence
$D^b(\Coh(\tilde{\Nilp}^{(1)}))\xrightarrow{\sim}D^b(\Acal\operatorname{-mod})$.

It turns out that the categories $\HC^G(\U)_0, \Ocat^{[0]}$ can be interpreted via $\Acal$.
Let $\pi:\tilde{\Nilp}^{(1)}\rightarrow \g^{*(1)}$ denote the natural map. So we have
a $G^{(1)}$-equivariant sheaf of algebras $\pi^*\Acal_\Ring$ over $\tilde{\Nilp}^{(1)}$
and consider the category $\Coh^{G^{(1)}}(\pi^*\Acal_\Ring)$.

\begin{Thm}\label{Thm:HCO_NCS}
We have equivalences of abelian categories
\begin{align}\label{eq:HC_NCS_equiv}&\HC^G(\U)_0\xrightarrow{\sim}
\Acal_\Ring\otimes_{\F[\g^{*(1)}]}\Acal^{opp}\operatorname{-mod}^{G^{(1)}},\\\label{eq:O_NCS_equiv}
&\Ocat^{[0]}\xrightarrow{\sim} \Coh^{G^{(1)}}(\pi^*\Acal_\Ring).
\end{align}
\end{Thm}

We would like to point out that the direct characteristic $0$ analog of
$\Coh^{G^{(1)}}(\pi^*\Acal_\Ring)$ has appeared before, in the paper \cite{BLin}.
In that paper the corresponding t-structure on the triangulated version
of the affine Hecke category (a characteristic $0$ analog of $D^b(\Ocat)$) was shown to coincide with the ``new'' t-structure
of Frenkel and Gaitsgory \cite{FG}. The equivalence (\ref{eq:O_NCS_equiv})
can be used to equip the heart of the new t-structure of the affine Hecke category
with a highest weight structure. In a subsequent paper we plan to use this fact to
prove a ``localization theorem'' for a category $\Ocat$ over a quantum group at a root of unity.

{\bf Acknowledgements}. I would like to thank Roman Bezrukavnikov, Gurbir Dhillon, Simon Riche, and Geordie Williamson for stimulating discussions. Mosto of this work was done during my participation in the special year on Geometric and Modular Representation Theory at the IAS in Spring 2021 were many topics related to this paper were discussed. I am grateful to the IAS and the program organizers and participants
for this valuable experience. My work was partially supported by the NSF
under grant DMS-2001139.

\section{Preliminaries: localization theorems}
In this section we recall the derived localization theorem from \cite{BMR}
and some related developments from \cite{BM}.
\subsection{Derived localization}\label{SS_der_loc}
Let $\tilde{\g}$ denote the Grothendieck-Springer resolution of
\begin{equation}\label{eq:notation_g_h}
\g^*_\h:=\g^*\otimes_{\h^*/W}\h^*.
\end{equation}
Set $\St_{\h}:=\tilde{\g}\times_{\g^*}\tilde{\g}$. We also consider the ``completed'' version
\begin{equation}\label{eq:notation_St_R}
\St_{\Ring}:=(\h^*/W)^{\wedge}\times_{\h^*/W}\St_\h
\end{equation}
(here and below $\bullet^\wedge$
denote the completion at $0$, e.g., in this particular case $(\h^*/W)^{\wedge}=
\Spec(\Ring^W)$). Note that $G$ naturally
acts on $\tilde{\g}, \St_{\h},\St_{\Ring}$. Note also that $\St_{\Ring}$ comes with a natural
morphism to $\h^{*\wedge}\times_{\h^{*\wedge}/W}\h^{*\wedge}$.

We can apply the Frobenius twist to all objects in the previous paragraph getting
schemes $\tilde{\g}^{(1)},\St_{\h}^{(1)},\St_{\Ring}^{(1)}$. They come with
an action of $G^{(1)}$. Note that the Artin-Schreier map $\h^*\rightarrow \h^{*(1)}$ is etale
hence identifies $\F[\h^{(1)*}]^{\wedge_0}$ with $\Ring$. In particular,
$\St^{(1)}_{\Ring}$ is still a scheme over $\h^{*\wedge}\times_{\h^{*\wedge}/W}\h^{*\wedge}$.

Consider the categories  $D^b(\Coh^{G}(\St^{(1)}_\h)), D^b(\Coh^{G}(\St_{\Ring}^{(1)}))$.
The categories are monoidal with respect to convolution of coherent sheaves. We can also consider the versions for $G^{(1)}$ instead of $G$, e.g., $D^b(\Coh^{G^{(1)}}(\St_{\Ring}^{(1)}))$.
We have a natural t-exact monoidal functor
$$D^b(\Coh^{G^{(1)}}(\St_{\Ring}^{(1)}))\rightarrow D^b(\Coh^{G}(\St_{\Ring}^{(1)}))$$
but it is not fully faithful (although its restriction to the hearts of default t-structures
is).
We have the completion (=pullback) functor
$$D^b(\Coh^{G}(\St^{(1)}_\h))\rightarrow D^b(\Coh^{G}(\St_{\Ring}^{(1)})),$$
it is monoidal.

We now recall a certain Azumaya algebra on
$$\tilde{\g}^{(1)}_\Ring:=\tilde{\g}^{(1)}\times_{\h^{(1)*}}\Spec(\Ring).$$

Let $U\subset B$ be the maximal unipotent and Borel subgroups of $G$. We will
write $\mathcal{B}$ for the flag variety $G/B$.
Consider the sheaf $D_{G/U}$
of differential operators on $G/U$. Let $\eta:G/U\twoheadrightarrow G/B$
be the projection. Consider the sheaf $\Dcal_{\h}:=(\eta_* D_{G/U})^T$.
This is an Azumaya algebra on $\tilde{\g}_\h^{(1)}\times_{\h^{*(1)}}\h^*$,
see, e.g., \cite[Sections 2.3, 3.1.3]{BMR}.
For $\mu\in \h^*$, consider the completion
\begin{equation}\label{eq:notation_D_mu}
\Dcal^{\wedge_\mu}:=\F[\h^*]^{\wedge_\mu}\otimes_{\F[\h^*]}\Dcal_{\h}.
\end{equation}
This completion can be viewed as an Azumaya algebra on $\tilde{\g}_{\Ring}^{(1)}$.

Set
\begin{equation}\label{eq:notation_U_mu}\U^{\wedge_\mu}:=\F[\h^*]^{\wedge_\mu}\otimes_{\F[\h^*/(W,\cdot)]}\U,
\end{equation}
where the quotient is taken for the dot-action (i.e., the $\rho$-shifted action)
of $W$.
By \cite[Proposition 3.4.1]{BMR}, $R\Gamma(\Dcal^{\wedge_\mu})=\U^{\wedge_\mu}$.
If $\mu$ is regular (for the dot action), then we have a category equivalence
\begin{equation}\label{eq:derived_loc_equiv}
R\Gamma: D^b(\Coh(\Dcal^{\wedge_\mu}))
\rightarrow D^b(\U^{\wedge_\mu}\operatorname{-mod}),
\end{equation}
 this is essentially \cite[Theorem 3.2]{BMR}.

The sheaves $\Dcal^{\wedge_\mu}$ for different $\mu$ with integral difference are Morita equivalent.
Now choose $\mu$ represented by a character of $T$ and, abusing the notation, denote the character also
by $\mu$. Then we have
a Morita equivalence between
$\Dcal^{\wedge_0}$ and $\Dcal^{\wedge_\mu}$  given by
\begin{equation}\label{eq:Morita_D_modules}
\Ocal_{\mathcal{B}}(\mu)\otimes_{\Ocal_{\mathcal{B}}}\bullet:\Coh(\Dcal^{\wedge_0})\rightarrow \Coh(\Dcal^{\wedge_\mu}).
\end{equation}
We also note that we have an algebra isomorphism
\begin{equation}\label{eq:D_alg_iso}
\Dcal^{\wedge_\mu}\cong \mathcal{O}_{\mathcal{B}}(\mu)\otimes \Dcal^{\wedge_0}\otimes
\mathcal{O}_{\mathcal{B}}(-\mu).
\end{equation}

 We will need the case of $\mu=-\rho$. We have a cover $\tilde{G}$ of $G$ (for example, the simply connected one) such that
$\rho$ is a weight for a maximal torus $\tilde{T}$ of $\tilde{G}$. The resulting Morita equivalence (\ref{eq:Morita_D_modules}) lifts to $\tilde{G}$-equivariant categories.

Now consider the algebra $\U^{\wedge_{-\rho}}$
and the sheaf $\Dcal^{\wedge_{-\rho}}$. It was shown in \cite[Proposition 5.2.1]{BMR} that
$\U^{\wedge_{-\rho}}$ is an Azumaya algebra on $\g^{*(1)}_{\Ring}$ and $\Dcal^{\wedge_{-\rho}}$
is obtained from $\U^{\wedge_{-\rho}}$ via pullback under the resolution morphism
$\tilde{\g}^{(1)}_{\Ring}\rightarrow \g^{*(1)}_{\Ring}$. It follows that the
restriction of $\Dcal^{\wedge_{-\rho}}\boxtimes \left(\Dcal^{\wedge_{-\rho}}\right)^{opp}$
to $\St_{\Ring}^{(1)}\subset \tilde{\g}^{(1)}_\Ring\times \tilde{\g}^{(1)}_\Ring$ (to be denoted by $\Dcal^{\wedge_{-\rho}}\boxtimes \left(\Dcal^{\wedge_{-\rho}}\right)^{opp}|_{\St}$) is $G$-equivariantly split with splitting bundle $\Ecal^{diag}_{-\rho}$
obtained by pulling back $\U^{\wedge_{-\rho}}$ under the morphism
$\St_{\Ring}^{(1)}\rightarrow \g^{*(1)}_{\Ring}$. Now note that
$\Dcal^{\wedge_{-\rho}}\boxtimes \left(\Dcal^{\wedge_{-\rho}}\right)^{opp}$
is Morita equivalent to $\Dcal^{\wedge_{0}}\boxtimes \left(\Dcal^{\wedge_{0}}\right)^{opp}$
via
\begin{equation}\label{eq:Morita_diagonal}
\Ocal_{\B\times \B}(\rho,-\rho)\otimes_{\Ocal_{\B\times\B}}\bullet:
\Coh^G(\Dcal^{\wedge_{-\rho}}\boxtimes \left(\Dcal^{\wedge_{-\rho}}\right)^{opp}|_{\St})
\xrightarrow{\sim} \Coh^G(\Dcal^{\wedge_{0}}\boxtimes \left(\Dcal^{\wedge_{0}}\right)^{opp}|_{\St}).
\end{equation}
While $\rho$ may fail to be
a character of $T$, the line bundle $\Ocal_{\B\times \B}(\rho,-\rho)$ is $G$-equivariant.

\begin{defi}\label{defi:Ecal}
Applying (\ref{eq:Morita_diagonal}) to $\Ecal^{diag}_{-\rho}$
we get a $G$-equivariant splitting bundle for the restriction
$\Dcal^{\wedge_{0}}\boxtimes \left(\Dcal^{\wedge_{0}}\right)^{opp}|_{\St}$, denote it
by $\Ecal^{diag}$.
\end{defi}

Now consider the scheme
\begin{equation}\label{eq:notation_g_wedge_g}
\tilde{\g}^{(1),\wedge_\g}:=\operatorname{Spec}(\F[[\g^{(1)*}]])\times_{\g^{(1)*}}\tilde{\g}^{(1)}_\Ring.
\end{equation}
The pullback of $\U^{\wedge_{-\rho}}$ to $\g^{*(1),\wedge_\g}_{\Ring}$ splits. Note that all splitting bundles for $\U^{\wedge_{-\rho}}$ are isomorphic (because it is a split Azumaya algebra over a complete local ring). Let $\Ecal'_{-\rho}$ denote the pullback of this splitting bundle to $\tilde{\g}^{(1),\wedge_\g}$. It is a splitting
bundle for the pullback $\Dcal^{\wedge_{-\rho,\g}}$ of $\Dcal^{\wedge_{-\rho}}$ to $\tilde{\g}^{(1),\wedge_\g}$. Applying a Morita equivalence analogous to (\ref{eq:Morita_D_modules})
to $\Ecal'_{-\rho}$ we get a splitting bundle for $\Dcal^{\wedge,0}$ to be denoted by $\Ecal'$.

\begin{Rem}\label{Rem:splitting_coincidence}
By the construction, the restrictions of $\Ecal_{diag}$ and $\Ecal'\boxtimes \Ecal'^*$ to
$$\St^{(1)}_\Ring\cap (\tilde{\g}^{(1),\wedge_\g}\times \tilde{\g}^{(1),\wedge_\g})$$
are isomorphic.
\end{Rem}

We proceed to a discussion of derived equivalences.
If $\mu$ is regular (for the dot action), then
$$R\Gamma: D^b(\Coh(\Dcal^{\wedge_\mu}))
\rightarrow D^b(\U^{\wedge_\mu}\operatorname{-mod})$$
is a category equivalence, this is essentially \cite[Theorem 3.2]{BMR}.
For example, we can take $\mu=0$. This gives rise to the following equivalences:
\begin{align}\label{eq:derived_equivalence_0}
&R\Gamma(\mathcal{E}'\otimes \bullet): D^b(\Coh(\tilde{\g}^{(1),\wedge_\g}))\xrightarrow{\sim}
D^b(\U^{\wedge_0,\g}\operatorname{-mod}),\\\label{eq:derived_equivalence_double}
&R\Gamma: D^b(\Coh^G(\Dcal^{\wedge_{0}}\boxtimes \left(\Dcal^{\wedge_{0}}\right)^{opp}))
\xrightarrow{\sim} D^b(\U^{\wedge_0}\otimes \left( \U^{\wedge_0}\right)^{opp}\operatorname{-mod}^G).
\end{align}
In (\ref{eq:derived_equivalence_0}) we set $\U^{\wedge_0,\g}:=\F[[\g^{(1)*}]]\otimes_{\F[\g^{(1)*}]}\U$.

Set
\begin{equation}\label{eq:notation_g_wedge}
\F[\g^{(1)*}]^\wedge:=\F[\g^{(1)*}]\otimes_{\F[\h^{(1)*}]^W}\F[[\h^{(1)*}]]^W.
\end{equation}
This algebra is the center of
$\U^{\wedge_0}$, this is an easy consequence of the Veldkamp theorem on the center
of $\U$.
Since $\U$ is flat over $\F[\g^{(1)*}]$ and $\St^{(1)}_\h$ is a complete intersection in
$\tilde{\g}^{(1)}\times \tilde{\g}^{(1)}$, (\ref{eq:derived_equivalence_double})
yields an equivalence
\begin{equation}\label{eq:RGamma_St_deformed}
R\Gamma: D^b(\Coh^G(\Dcal^{\wedge_{0}}\boxtimes \left(\Dcal^{\wedge_{0}}\right)^{opp}|_{\St}))
\xrightarrow{\sim} D^b(\U^{\wedge_0}\otimes_{\F[\g^{(1)*}]^\wedge}\left(\U^{\wedge_0}\right)^{opp}
\operatorname{-mod}^G),
\end{equation}
see the discussion of exact fiber products and base changes in \cite[Sections 1.3-1.5]{BM}.
Hence we have an equivalence
\begin{equation}\label{eq:der_loc_St}
R\Gamma(\Ecal_{diag}\otimes\bullet):D^b(\Coh^{G}\St_{\Ring}^{(1)})\xrightarrow{\sim}
D^b(\U^{\wedge_0}\otimes_{\F[\g^{(1)*}]^\wedge}\left(\U^{\wedge_0}\right)^{opp}
\operatorname{-mod}^G).
\end{equation}
We note that (\ref{eq:der_loc_St}) is an $\Ring$-bilinear monoidal equivalence (with respect to
the covolution vs tensor product of bimodules).

\begin{Rem}\label{Rem:specialized_equiv}
Consider the specialization $\U_0$ of $\U^{\wedge_0}$ to the closed point of $\Ring$
as well as the Steinberg variety $\St:=\St_\Ring\times_{\operatorname{Spec}(\Ring)}\operatorname{pt}$.
For the same reason as for (\ref{eq:der_loc_St}) we have
\begin{equation}\label{eq:der_loc_St_specialized}
R\Gamma(\Ecal_{diag}\otimes\bullet):D^b(\Coh^{G}\St^{(1)})\xrightarrow{\sim}
D^b(\U\otimes_{\F[\g^{(1)*}]}\U_0^{opp}
\operatorname{-mod}^G).
\end{equation}
We note that it is an equivalence of left module categories over
the equivalent monoidal categories in (\ref{eq:der_loc_St}).
\end{Rem}

\subsection{Tilting bundle and noncommutative Springer resolution}\label{SS:tilt_NC}
We will need the construction of a tilting bundle on $\tilde{\g}^{(1)}$ from \cite{BM}.
On $\tilde{\g}^{(1)}$ we have a $G^{(1)}\times \mathbb{G}_m$-equivariant
vector bundle $\Tcal_\h$ that is a
tilting generator meaning that the following two conditions hold:
\begin{itemize}
\item
$\Ext^i(\Tcal_\h,\Tcal_\h)=0$ for $i>0$,
\item and the algebra
$\Acal_\h:=\End(\Tcal_\h)$ has finite homological dimension.
\end{itemize}
This bundle was constructed in
\cite{BM}, see \cite[Theorem 1.5.1]{BM} for the statement and \cite[Section 2.5]{BM}
for the construction of $\Tcal_\h$. More precisely, \cite{BM} introduces a vector
bundle on $\tilde{\g}^{(1)}$ denoted there by $\Ecal$. The relation between the two vector bundles is
as follows: $\Tcal_\h=\Ecal^*$.

The functor $R\Gamma(\Tcal_\h\otimes\bullet)$
is an equivalence $D^b(\Coh(\tilde{\g}^{(1)}))\xrightarrow{\sim}
D^b(\Acal_\h\operatorname{-mod})$, see \cite[Section 1.5.3]{BM}.

Consider the algebras
$$\Acal_{\Ring}:=\Acal_\h\otimes_{\F[\h^{(1)*}]}\Ring,
\quad \Acal^{diag}_{\Ring}:=\Acal_{\Ring}\otimes_{\F[\g^{*(1)}]^\wedge}\Acal_{\Ring}^{opp},$$
The group $G^{(1)}$ acts on $\Acal_{\Ring}, \Acal_{\Ring}^{diag}$ by $\Ring$-algebra automorphisms. Note that,
similarly to (\ref{eq:der_loc_St}), we
have equivalences
\begin{equation}\label{eq:derived_equiv_bimod}
D^b(\Coh^{G^{(1)}}\St_{\Ring}^{(1)})\xrightarrow{\sim}
D^b(\Acal^{diag}_{\Ring}\operatorname{-mod}^{G^{(1)}}),
\quad D^b(\Coh^{G}\St_{\Ring}^{(1)})\xrightarrow{\sim}
D^b(\Acal^{diag}_{\Ring}\operatorname{-mod}^{G})
\end{equation}
given by $R\Gamma([\Tcal_{\Ring}\boxtimes \Tcal^*_{\Ring}]|_{\St}\otimes\bullet)$, where
we write $\Tcal_{\Ring}$ for $\Tcal\otimes_{\F[\h^{(1)*}]}\Ring$. Note that the categories
$\Acal^{diag}_{\Ring}\operatorname{-mod}^{G^{(1)}}, \Acal^{diag}_{\Ring}\operatorname{-mod}^{G}$ are monoidal with respect
to the functor $\bullet\otimes_{\Acal_{\Ring}}\bullet$. The equivalences
(\ref{eq:derived_equiv_bimod}) are monoidal. They are also $\Ring$-bilinear.

We will need to relate $\mathcal{T}_\h$ to the splitting bundle arising from $\mathcal{D}^{\wedge_0}$.
The following was proved in \cite{BM} (somewhat implicitly, see \cite[Lemma 4.7]{BL} for
an explicit proof).

\begin{Lem}\label{Lem:same_summands}
The bundle $\Tcal_\h^{\wedge_\g}$, the restriction of $\Tcal_\h$ to $\tilde{\g}^{(1),\wedge_\g}$, has the same indecomposable direct summands
as $\Ecal'$.
\end{Lem}

Set $\Acal:=\Acal_\Ring\otimes_\Ring \F$. Similarly to Remark \ref{Rem:specialized_equiv},
we have a derived equivalence

\begin{equation}\label{eq:derived_equiv_bimod_specialized}
D^b(\Coh^{G^{(1)}}\St^{(1)})\xrightarrow{\sim}
D^b(\Acal_\Ring\otimes_{\F[\g^{*(1)}]} \Acal^{opp}\operatorname{-mod}^{G^{(1)}}).
\end{equation}

\section{Preliminaries: Soergel and Harish-Chandra bimodules}
In this section we  mostly review constructions and results
from \cite{BR}. We also discuss their connection with the constructions
from the previous section.
\subsection{Soergel bimodules following \cite{BR}}\label{SS_BR_AS}
One family of categories introduced in \cite{BR} has to do with Abe's construction of the category
of Soergel bimodules. First, some notation. 
For a right $\F[\h^{(1)*}]$-module $M$ we write
\begin{equation}\label{eq:module_functors}
M_{loc}:=M\otimes_{\F[\h^{(1)*}]}\F(\h^{(1)*}), M^\wedge:=M\otimes_{\F[\h^{(1)*}]}\Ring,
\end{equation}
where, recall, $\Ring$ stands for the completion of $\F[\h^{*(1)}]$ at zero.

Following \cite[Section 2.2]{BR}, we consider the category
$\mathsf{C}'_{ext}$ whose objects are pairs consisting of
\begin{itemize}
\item[(i)] graded $\F[\h^{(1)*}]$-bimodules $B$ and
\item[(ii)] decompositions $$B_{loc}=\bigoplus_{x\in W^{ea}}B^x_{loc},$$
into the direct sum of $\F(\h^{(1)*})$-subspaces such that there are only finitely many nonzero summands and
$br=x(r)b$ for all $b\in B^x_{loc}$ and $r\in \F[\h^{(1)*}]$.
\end{itemize}
The morphisms in $\mathsf{C}'_{ext}$ are graded bimodule homomorphisms $\varphi: B\rightarrow B'$
such that $\varphi_{loc}(B^x_{loc})\subset B'^x_{loc}$ for all $x\in W^{ea}$. Inside $\mathsf{C}'_{ext}$
we consider the full subcategory $\mathsf{C}_{ext}$ consisting of all objects that are finitely
generated as bimodules and also flat as right $\F[\h^{(1)*}]$-modules. Similarly to \cite[Lemma 2.6]{Abe},
the objects in $\mathsf{C}_{ext}$ are finitely generated both as left and as right $\F[\h^{(1)*}]$-modules.
In particular, they are free as right $\F[\h^{(1)*}]$-modules. The category $\mathsf{C}'_{ext}$
has a natural monoidal structure lifting $\bullet\otimes_{\F[\h^{(1)*}]}\bullet$. The subcategory
$\mathsf{C}_{ext}$ is an additive monoidal $\F[\h^{*(1)}]$-bilinear subcategory. Note also that we have the grading shift
endo-functor, to be denoted by $\langle 1\rangle$, of $\mathsf{C}'_{ext}$ that preserves
$\mathsf{C}_{ext}$.

We will need two families of objects in $\mathsf{C}_{ext}$. First, there are standard objects,
$\Delta_x, x\in W^{ea}$. Each $\Delta_x$ is a rank one free right $\F[\h^{(1)*}]$-module
(with generator in degree $0$),
where the left action is introduced via twist with $x$ so that $\Delta_{x,loc}=\Delta_{x,loc}^x$.
For a simple affine reflection $s\in W^a$ we also have the Bott-Samelson object $B_s$
whose underlying graded bimodule is $\F[\h^{(1)*}]\otimes_{\F[\h^{(1)*}]^s}\F[\h^{(1)*}]$ with the natural
decomposition for the localized bimodule. Let $\SBim$ denote the full subcategory
of $\mathsf{C}_{ext}$ generated by $\Delta_x, x\in \Lambda/\Lambda_r$, and $B_s$, where
$s$ runs over the set of simple affine reflections  under the operations of taking tensor products, direct sums/summands
and grading shifts. It is known, essentially after \cite{Abe} (see \cite[Theorem 1.1]{Abe} for the case when $W^{ea}=W^a$), that the indecomposables
in $\SBim$ are labelled by the elements of $W^{ea}$. Namely, for $w\in W^{ea}$, we can write
a reduced expression $w=xs_1\ldots s_k$ with $x\in \Lambda/\Lambda_r$ and $s_i$ simple affine reflections.
Then the indecomposable object $B_w$ is a direct summand  in
$$\Delta_x\otimes_{\F[\h^{(1)*}]}B_{s_1}\otimes_{\F[\h^{(1)*}]}B_{s_2}\otimes_{\F[\h^{(1)*}]}\ldots
\otimes_{\F[\h^{(1)*}]}B_{s_k}$$
and all other direct summands are of the form $B_u$ with $u<w$ in the Bruhat order
(and some grading shifts).

We will work not with $\SBim,\mathsf{C}_{ext}$ and $\mathsf{C}'_{ext}$ but with their completed
(and ungraded versions). Namely, we consider the category $\mathsf{C}'^\wedge_{ext}$
consisting of $\F[[\h^{(1)*}]]$-bimodules with additional structure as in (ii) above.
Inside there is the full subcategory $\mathsf{C}^\wedge_{ext}\subset \mathsf{C}'^\wedge_{ext}$
defined similarly to $\mathsf{C}_{ext}$. Note that we have the completion functor $\bullet^\wedge:\mathsf{C}'_{ext}\rightarrow \mathsf{C}'^\wedge_{ext}$ defined by (\ref{eq:module_functors}).
It is exact and monoidal. It sends $\mathsf{C}_{ext}$ to $\mathsf{C}^\wedge_{ext}$.
We define $\SBim^\wedge$ as the full subcategory of $\mathsf{C}^\wedge_{ext}$ generated by
$B_s^\wedge$ for simple affine reflections $s$ and $\Delta_x^\wedge$ for $x\in \Lambda/\Lambda_r$.
Below we will write $B_s^{AS}$ for $B_s^\wedge$ and $\Delta_x^{AS}$ for $\Delta_x^\wedge$
(for all $x\in W^{ea}$).

The following lemma describes basic properties of the category $\SBim^{\wedge}$.

\begin{Lem}\label{Lem:SBim_completed}
The following claims are true:
\begin{enumerate}
\item For $B_1,B_2\in \mathsf{C}_{ext}$ we have
$$
\left(\bigoplus_{i\in \Z}\Hom_{\mathsf{C}_{ext}}(B_1,B_2\langle i\rangle)\right)^\wedge\xrightarrow{\sim}\Hom_{\mathsf{C}^\wedge_{ext}}(B_1^\wedge,B_2^\wedge).
$$
\item The indecomposable objects in $\SBim^\wedge$ are precisely the objects $B_x^\wedge$ for
$x\in W^{ea}$.
\end{enumerate}
\end{Lem}
\begin{proof}
Note that $\bigoplus_{i\in \Z}\Hom_{\mathsf{C}_{ext}}(B_1,B_2\langle i\rangle)$ embeds into
$\Hom_{\F[\h^{(1)*}]\operatorname{-bimod}}(B_1,B_2)$ with the image consisting of all $\varphi$
such that $\varphi_{loc}$ preserves the decomposition in (ii) above. Since $B_1$ is finitely
generated as an $\F[\h^{(1)*}]$-bimodule, we have
$$\Hom_{\F[\h^{(1)*}]\operatorname{-bimod}}(B_1,B_2)^\wedge\xrightarrow{\sim}
\Hom_{\F[[\h^{(1)*}]]\operatorname{-bimod}}(B_1^\wedge,B_2^\wedge).$$
We need to show that this isomorphism intertwines the sub-bimodules of all maps $\varphi$
such that $\varphi_{loc}$ intertwine the decompositions analogous to those in (ii).
The inclusion
\begin{equation}\label{eq:Hom_inclusion}\left(\bigoplus_{i\in \Z}\Hom_{\mathsf{C}_{ext}}(B_1,B_2\langle i\rangle)\right)^\wedge\subset\Hom_{\mathsf{C}^\wedge_{ext}}(B_1^\wedge,B_2^\wedge).
\end{equation}
follows from the observation that the functor $\bullet\otimes_{\F(\h^{(1)*})}\F((\h^{(1)*}))$
preserves the decompositions like in (ii). On the hand, the preimage of
$$\Hom_{\mathsf{C}^\wedge_{ext}}(B_1^\wedge,B_2^\wedge)\subset
\Hom_{\F[[\h^{(1)*}]]\operatorname{-bimod}}(B_1^\wedge,B_2^\wedge)$$
inside  $\Hom_{\F[\h^{(1)*}]\operatorname{-bimod}}(B_1,B_2)$ consists of maps
that preserve the decompositions in (ii). It follows that (\ref{eq:Hom_inclusion}) is
an equality.

(2) is proved exactly as \cite[Lemma 6.9]{EL}.
\end{proof}

To relate $\SBim$ (or $\SBim^\wedge$) to a category of Harish-Chandra bimodules (that will be explained in
the next section) we need an intermediate
geometric category considered in \cite{BR}. Define the scheme
\begin{equation}\label{eq:Y_scheme}
Y:=\h^{(1)*}\times_{\h^{(1)*}/W}\h^{(1)*}.
\end{equation}

There is a certain affine group  scheme $\mathfrak{J}$ over
$Y$ introduced in \cite[Section 2.3]{BR}. Namely, let $\g^{(1)*,reg}$ denote the locus of
regular (not necessarily semisimple) elements in $\g^{(1)*}$. Inside $\g^{(1)*,reg}$
we have the Kostant-Slodowy slice $S^{(1)}$. The quotient morphism $\g^{(1)*}
\rightarrow \g^{(1)*}/\!/ G^{(1)}\cong \h^{(1)*}/W$ restricts to an isomorphism $S^{(1)}\xrightarrow{\sim}\h^{(1)*}/W$.
Over $\g^{(1)*,reg}$ we have the universal centralizer group scheme,
to be denoted here by $\mathfrak{C}^{(1)}$. Over $s\in S^{(1)}$, the fiber of
$\mathfrak{C}^{(1)}$ is the stabilizer of $s$ in $G^{(1)}$. We can view $\mathfrak{C}^{(1)}$
as a group scheme over  $S^{(1)}\xrightarrow{\sim}\h^{(1)*}/W$. Then we set
$$\mathfrak{J}=\h^{(1)*}\times_{\h^{(1)*}/W}\mathfrak{C}^{(1)}\times_{\h^{(1)*/W}}\h^{(1)*}.$$

We note that, by the construction of $\mathfrak{J}$, the multiplicative group
$\F^\times$ acts on $\mathfrak{J}$ by automorphisms compatibly with the dilation
action on $\h^{(1)*}\times_{\h^{(1)*}/W}\h^{(1)*}$. So it makes sense to
consider the category $\operatorname{Rep}_{fl}^{\F^\times}(\mathfrak{J})$ of
$\F^\times\ltimes\mathfrak{J}$-equivariant coherent sheaves on $Y$ that are
flat over the second copy of $\h^{(1)*}$. As explained in \cite[Section 2.3]{BR},
$\operatorname{Rep}_{fl}^{\F^\times}(\mathfrak{J})$ is a monoidal category.

We will need the completed version of this category. Consider the scheme
$Y^\wedge:=Y\times_{\h^{(1)*}}\h^{*\wedge}=\h^{*\wedge}\times_{\h^{*\wedge}/W}\h^{*\wedge}$
(where we write $\h^{*\wedge}$ for $\operatorname{Spec}(\Ring)\cong \operatorname{Spec}(\F[[\h^{(1)*}]])$)
and the pullback $\mathfrak{J}^\wedge$ of $\mathfrak{J}$ to $Y^\wedge$.
So we can consider the category $\operatorname{Rep}_{fl}(\mathfrak{J}^\wedge)$.
It is monoidal similarly to \cite[Section 2.3]{BR}. We have the completion functor $\bullet^\wedge:
\operatorname{Rep}_{fl}^{\F^\times}(\mathfrak{J})\rightarrow
\operatorname{Rep}_{fl}(\mathfrak{J}^\wedge)$ defined as above.

\begin{Lem}\label{Lem:full_embedding_geometric}
We have a monoidal $\Ring$-bilinear full embedding  $\SBim^\wedge\hookrightarrow
\operatorname{Rep}_{fl}(\mathfrak{J}^\wedge)$.
\end{Lem}
\begin{proof}
The proof is very similar to that of \cite[Theorem 2.10]{BR}. Namely,
\cite[Proposition 2.7]{BR}, we have a fully faithful monoidal embedding
\begin{equation}\label{eq:embedding2}
\operatorname{Rep}_{fl}^{\F^\times}(\mathfrak{J})\hookrightarrow
\mathsf{C}_{ext}.
\end{equation}
We claim that it gives rise to a full monoidal embedding
\begin{equation}\label{eq:embedding3}
\operatorname{Rep}_{fl}(\mathfrak{J}^\wedge)\hookrightarrow
\mathsf{C}^\wedge_{ext}.
\end{equation}
The functor is constructed similarly to the proof \cite[Proposition 2.7]{BR},
it is monoidal. To show it is full we need to relate Hom's in the categories $\operatorname{Rep}_{fl}^{\F^\times}(\mathfrak{J}), \operatorname{Rep}_{fl}(\mathfrak{J}^\wedge)$
and also in $\operatorname{Rep}_{fl}(\mathfrak{J})$.
They will be denoted by $\Hom_{\F^\times,\mathfrak{J}}, \Hom_{\mathfrak{J}^\wedge},
\Hom_{\mathfrak{J}}$. Namely, thanks to (1) of Lemma \ref{Lem:SBim_completed}
and \cite[Proposition 2.7]{BR}, it is enough to
prove that
\begin{equation}\label{eq:hom_equality}
\left(\bigoplus_{i\in \Z}\Hom_{\F^\times,\mathfrak{J}}(\mathcal{F}_1,\mathcal{F}_2\langle i\rangle)\right)^\wedge\xrightarrow{\sim}\Hom_{\mathfrak{J}^\wedge}(\mathcal{F}_1^\wedge,\mathcal{F}_2^\wedge)
\end{equation}
for all $\mathcal{F}_1,\mathcal{F}_2\in \operatorname{Rep}_{fl}^{\F^\times}(\mathfrak{J})$.
Note that the direct sum in the  left hand side of (\ref{eq:hom_equality}) is nothing else but $\Hom_{\mathfrak{J}}(\mathcal{F}_1,\mathcal{F}_2)$. As in the proof of (1) of Lemma
\ref{Lem:SBim_completed}, we have
$$\Hom_{\Coh(Y)}(\mathcal{F}_1,\mathcal{F}_2)^\wedge\xrightarrow{\sim}
\Hom_{\Coh(Y^\wedge)}(\mathcal{F}_1^\wedge,\mathcal{F}_2^\wedge).$$
We need to show that this isomorphism intertwines $\Hom_{\mathfrak{J}}(\mathcal{F}_1,\mathcal{F}_2)^\wedge$ with $ \Hom_{\mathfrak{J}^\wedge}(\mathcal{F}_1^\wedge,\mathcal{F}_2^\wedge)$.

Note that $\Hom_{\mathfrak{J}}(\mathcal{F}_1,\mathcal{F}_2)$ is the $\F[Y]$-submodule
of solutions to a finite system of linear equations with coefficients in $\F[Y]$.
Indeed, $\Hom_{\Coh(Y)}(\mathcal{F}_1,\mathcal{F}_2)$ is a $\mathfrak{J}$-module,
equivalently, an $\F[\mathfrak{J}]$-comodule. If
$\epsilon:\F[Y]\rightarrow \F[\mathfrak{J}]$ denotes the  unit map
and $$\alpha:\Hom_{\Coh(Y)}(\mathcal{F}_1,\mathcal{F}_2)\rightarrow
\Hom_{\Coh(Y)}(\mathcal{F}_1,\mathcal{F}_2)\otimes_{\F[Y]}\F[\mathfrak{J}]$$
is the co-action map, then
$$
\Hom_{\mathfrak{J}}(\mathcal{F}_1,\mathcal{F}_2)=\{\varphi\in \Hom_{\Coh(Y)}(\mathcal{F}_1,\mathcal{F}_2)| \alpha(\varphi)=\varphi\otimes \epsilon\}.$$
The equation in the right hand side translates to a finite system of linear equations on $\varphi$
with coefficients in $\F[Y]$, which establishes the claim in the beginning of the paragraph.
Now note that
$$\Hom_{\mathfrak{J}^\wedge}(\mathcal{F}_1^\wedge,\mathcal{F}_2^\wedge)\subset
\Hom_{\Coh(Y^\wedge)}(\mathcal{F}_1^\wedge,\mathcal{F}_2^\wedge)$$
is specified by the same equations. This finishes the proof of (\ref{eq:hom_equality})
and hence the construction of (\ref{eq:embedding3}).

To establish the full embedding $\SBim^\wedge\hookrightarrow
\operatorname{Rep}_{fl}(\mathfrak{J}^\wedge)$ it suffices
show that the essential image of (\ref{eq:embedding3})
contains $\SBim^\wedge$. This follows from
\cite[Lemma 2.9]{BR}.
\end{proof}

\begin{Rem}\label{Rem:full_embedding_construction}
We will need a construction of (\ref{eq:embedding3}) following the proof of
\cite[Proposition 2.7]{BR}. It is essentially a (partially) forgetful functor. An object $\mathcal{F}$ in
$\operatorname{Rep}_{fl}(\mathfrak{J}^\wedge)$ is, in particular, an $\Ring$-bimodule
that is flat over the second copy of $\Ring$. The localization of $\mathcal{F}$ to the regular
locus in $\h^{(1)*\wedge}/W$ carries an action of $T$. The corresponding locus in $Y$ splits
into the disjoint union of components indexed by $W$. These two structures give the decomposition
of $\mathcal{F}_{loc}$ as in (ii) in the beginning of the section.
\end{Rem}

\subsection{Harish-Chandra bimodules following \cite{BR}}\label{SS_BR_HC}
Here we recall that category of modular HC bimodules following \cite{BR} and explain a version of
the main result of \cite{BR}.

We start by recalling basics following \cite[Section 3]{BR}.
Recall that $\U^{\wedge_0}$ stands for the partial completion of $\U$ at HC central character $0$:
$$\U^{\wedge_0}=\U\otimes_{\F[\h^*/(W,\cdot)]}\F[\h^*/(W,\cdot)]^{\wedge_0}.$$
This is an $\Ring$-algebra with a $G$-action. The center of $\U^{\wedge_0}$ is $\F[\g^{(1)*}]^\wedge$,
see Section \ref{SS_der_loc}.

Consider the category $\U^{\wedge_0,opp}\operatorname{-mod}^G$ of (weakly) $G$-equivariant
right $\U^{\wedge_0}$-modules. Such a module carries a natural left $\U$-action commuting
with the right $\U^{\wedge_0}$-action that is uniquely recovered from the condition that the
adjoint $\g$-action coincides with the differential of the $G$-action (this condition can be expressed
by saying that we get strongly $G$-equivariant $\U$-bimodules). The objects in
$\U^{\wedge_0,opp}\operatorname{-mod}^G$ will be called {\it Harish-Chandra bimodules}.

For $\mu\in \Lambda$, let $\HC^G(\U^{\wedge_\mu}\operatorname{-}\U^{\wedge_0})$ stand for the full subcategory $\U^{\wedge_0,opp}\operatorname{-mod}^G$ of all objects where the left action of
$\U$ factors through $\U^{\wedge_\mu}$.
We have the decomposition,
see e.g. \cite[(3.13)]{BR}
\begin{equation}\label{eq:category_decomposition}
\U^{\wedge_0,opp}\operatorname{-mod}^G=\bigoplus_{\mu} \HC^G(\U^{\wedge_\mu}\operatorname{-}\U^{\wedge_0}),
\end{equation}
where the summation is taken over all orbits of $W$ in the $\mathbb{F}_p$-points in $\h^*$. We will be primarily interested in the category
$\HC^G(\U^{\wedge_0}):=\HC^G(\U^{\wedge_0}\operatorname{-}\U^{\wedge_0})$. It is $\Ring$-bilinear and monoidal with respect to $\bullet\otimes_{\U^{\wedge_0}}\bullet$. We write $\operatorname{pr}_0$ for the projection
functor from $\U^{\wedge_0,opp}\operatorname{-mod}^G$ to $\HC^G(\U^{\wedge_0})$.

Here is a special class of bimodules considered in \cite{BR}. We say that an object in
$\HC^G(\U^{\wedge_0})$ is {\it diagonally induced} if it is a direct summand in
$V\otimes \U^{\wedge_0}$ for some finite dimensional representation $V$ of $G$.
This is a monoidal subcategory of $\HC^G(\U^{\wedge_0})$ to be denoted by
$\HC^G_{diag}(\U^{\wedge_0})$.

It was essentially shown in \cite{BR} that there is a full embedding of
$\SBim^\wedge$ into $\HC^G_{diag}(\U^{\wedge_0})$. One can describe the images of
$\Delta_x^{AS}, B_s^{AS}\in \SBim$ in   $\HC^G_{diag}(\U^{\wedge_0})$ for all $x\in \Lambda/\Lambda_r$
and all simple affine reflections $s$. First, we have the reflection
bimodules $B_s^{HC}\in \HC^G(\U^{\wedge_0})$ for all simple affine reflections $s$, in the notation of \cite{BR}
these are the objects
$$\mathsf{P}^{0,\mu_s}\widehat{\otimes}_{\mathcal{U}\g}\mathsf{P}^{\mu_s,0}.$$
Tensoring with $B_s^{HC}$ on the left gives the classical reflection functor
corresponding to $s$, it will be denoted by $\Theta_s$.
For $x\in \Lambda/\Lambda_r,$ we consider the standard objects $\Delta_x^{HC}$
for $x\in \Lambda/\Lambda_r$, the translation bimodules from $0$ to $x^{-1}\cdot 0$
(where we write $\cdot$ to indicate the $p$-scaled action of $W^{ea}$ on $\Lambda$,
note that the condition that $x\in \Lambda/\Lambda_r$ is equivalent to $x^{-1}\cdot 0$
and $0$ being in the same alcove). The full embedding $\SBim^\wedge\hookrightarrow \HC^G_{diag}(\U^{\wedge_0})$ essentially constructed in \cite{BR} (see Section 6 there
for the final construction) sends $B_s^{AS}$ to $B_s^{HC}$ and $\Delta_x^{AS}$ to
$\Delta_x^{HC}$.

Let us recall the construction of the full embedding. Recall the schemes $Y,Y^\wedge$
from Section \ref{SS_BR_AS} as well as the group scheme $\mathfrak{J}$ on $Y$ and
its restriction $\mathfrak{J}^\wedge$ to $Y^{\wedge}$. Consider another group scheme,
$\mathfrak{I}$, on $Y$ defined analogously to $\mathfrak{J}$ but for the action of
$G$ on $\g^{(1)*}$ instead of the action of $G^{(1)}$ (this group scheme is denoted by
$\mathbb{I}$ in \cite{BR}). We have a natural epimorphism $\mathfrak{I}\rightarrow
\mathfrak{J}$ whose kernel is the constant group scheme on $Y$ with fiber $G_1$, the first Frobenius kernel. Let $\mathfrak{I}^\wedge$ denote the pullback of $\mathfrak{I}$ to $Y^{\wedge}$.
We can consider the category $\operatorname{Rep}(\mathfrak{I}^\wedge)$ defined similarly
to $\operatorname{Rep}(\mathfrak{J}^\wedge)$. We have a full embedding
$\operatorname{Rep}(\mathfrak{J}^\wedge)\hookrightarrow \operatorname{Rep}(\mathfrak{I}^\wedge)$,
the pullback via $\mathfrak{I}^\wedge\twoheadrightarrow \mathfrak{J}^\wedge$.

We will also need to enlarge $\HC^G(\U^{\wedge_0})$.
By definition,
the enlarged category we need
consists of all $G$-equivariant $\U^{\wedge_0}$-bimodules where the left and the right
actions of the $p$-center $\F[\g^{(1)*}]^\wedge$ coincide. Denote this category
by $\overline{\HC}^G(\U^{\wedge_0})$. Since the left and right actions of the $p$-center
on every HC bimodule coincide,
we have an $\Ring$-bilinear  full monoidal inclusion $
\HC^G(\U^{\wedge_0})\hookrightarrow \overline{\HC}^G(\U^{\wedge_0})$,
see \cite[(3.6)]{BR}.

The construction in \cite{BR} starts with producing a functor
\begin{equation}\label{eq:functor_HC_slice}
\overline{\HC}^G(\U^{\wedge_0})\rightarrow \operatorname{Rep}(\mathfrak{I}^\wedge)
\end{equation}
that restricts to
a fully faithful embedding $\HC^G_{diag}(\U^{\wedge_0})\hookrightarrow \operatorname{Rep}_{fl}(\mathfrak{I}^\wedge)$. The construction is as follows. Consider the
locus
$$\St_{\Ring}^{(1),reg}:=\g^{(1)*,reg}\times_{\g^{(1)*}}\St^{(1)}_\Ring\subset \St_\Ring^{(1)}$$
Note that the restriction of the projective morphism
$$\St_{\Ring}^{(1)}\rightarrow
\operatorname{Spec}(\Ring)\times_{\h^{(1)*}/W}\g^{(1)*}\times_{\h^{(1)*}/W} \operatorname{Spec}(\Ring)$$
to $\St_{\Ring}^{(1),reg}$ is an open embedding. We can view $\U^{\wedge_0}\otimes_{\F[\g^{(1)*}]^\wedge}\U^{\wedge_0,opp}$ as a sheaf of algebras on
$\operatorname{Spec}(\Ring)\times_{\h^{(1)*}/W}\g^{(1)*}\times_{\h^{(1)*}/W} \operatorname{Spec}(\Ring)$.
Its restriction to $\St_{\Ring}^{(1),reg}$, to be denoted by
$\left(\U^{\wedge_0}\otimes_{\F[\g^{(1)*}]^\wedge}\U^{\wedge_0,opp}\right)^{reg}$, coincides with the restriction of $\Dcal^{\wedge_{0}}\boxtimes \left(\Dcal^{\wedge_{0}}\right)^{opp}|_{\St}$. In particular,
$\left(\U^{\wedge_0}\otimes_{\F[\g^{(1)*}]^\wedge}\U^{\wedge_0,opp}\right)^{reg}$ is a $G$-equivariant
Azumaya algebra with a $G$-equivariant splitting bundle $\mathcal{E}_{diag}^{reg}$, the restriction of the splitting bundle $\Ecal_{diag}$
introduced in Section \ref{SS_der_loc} to $\St_\Ring^{(1),reg}$.

The functor (\ref{eq:functor_HC_slice}) is constructed as follows. Start with an object in
$\overline{\HC}^G(\U^{\wedge_0})$ and restrict it to $\St_\Ring^{(1),reg}$. Applying the
equivalence coming from the splitting bundle $\Ecal_{diag}^{reg}$, we get a $G$-equivariant
coherent sheaf on  $\St_\Ring^{(1),reg}$. Note that $Y^\wedge$ embeds into
$\St_\Ring^{(1),reg}$, this embedding is induced by the embedding of $S$ into $\g^{(1)*}$.
So we can restrict a $G$-equivariant coherent sheaf from $\St_\Ring^{(1),reg}$ to $Y^\wedge$
getting an $\mathfrak{I}^\wedge$-equivariant coherent sheaf. For (\ref{eq:functor_HC_slice})
we take the resulting composition:
\begin{equation}\label{eq:functor_composition}
\overline{\HC}^G(\U^{\wedge_0})\rightarrow
\Coh^G(\left(\U^{\wedge_0}\otimes_{\F[\g^{(1)*}]^\wedge}\U^{\wedge_0,opp}\right)^{reg})
\xrightarrow{\sim} \Coh^G(\St_\Ring^{(1),reg})\rightarrow \operatorname{Rep}(\mathfrak{I}^\wedge).
\end{equation}
In fact, later we will see that the image of $\HC^G(\U^{\wedge_0})$ in
$\Coh^G(\St_\Ring^{(1),reg})$ lies in $\Coh^{G^{(1)}}(\St_\Ring^{(1),reg})$.

We note that thus constructed functor (\ref{eq:functor_HC_slice}) is the same as
the composition of \cite[(3.18)]{BR} and the inverse of the equivalence $\mathcal{L}_{0,0}$
in \cite[Corollary 4.8]{BR}. Indeed, to get from $\Coh^G(\left(\U^{\wedge_0}\otimes_{\F[\g^{(1)*}]^\wedge}\U^{\wedge_0,opp}\right)^{reg})$
to $\operatorname{Rep}(\mathfrak{I}^\wedge)$ we first apply the equivalence coming from the
splitting bundle, and then restrict to $Y^\wedge$. Bezrukavnikov and Riche first restrict to
$Y^\wedge$ and then apply the equivalence coming from the splitting bundle constructed in
\cite[Theorem 4.3]{BR}. To see that our construction agrees with that from \cite{BR}
it remains to prove the following lemma.

\begin{Lem}\label{Lem:splitting_coincidence}
The splitting bundle from \cite[Theorem 4.3]{BR} coincides with the restriction of
$\mathcal{E}_{diag}$ to $Y^\wedge$.
\end{Lem}
\begin{proof}
We write $L(\rho)$ for the irreducible $G$-module with highest weight $\rho$.
In the notation of \cite{BR}, their splitting bundle is given by the restriction
to $Y^\wedge$ of $\mathsf{P}^{0,-\rho}\widehat{\otimes}_{\U}\mathsf{P}^{-\rho,0}$,
where $\mathsf{P}^{0,-\rho}$ is given as the projection to $\HC^G(\U^{\wedge_0}\operatorname{-}\U^{\wedge_{-\rho}})$
of $\U^{\wedge_0}\otimes L(\rho)$ (where $\U^{\wedge_0}$ acts on this bimodule
from the left as on the direct sum of the $\dim L(\rho)$ copies of the regular module). Similarly, $\mathsf{P}^{-\rho,0}$ is the projection of
$L(\rho)\otimes\U^{\wedge_0}$ to $\HC^G(\U^{\wedge_{-\rho}}\operatorname{-}\U^{\wedge_{0}})$.

Note that  $\U^{\wedge_0}\otimes L(\rho)=\Gamma(\mathcal{D}^{\wedge_0}\otimes L(\rho))$.
We view $\mathcal{M}:=\mathcal{D}^{\wedge_0}\otimes L(\rho)$ as a weakly $G$-equivariant
left $\mathcal{D}^{\wedge_0}$-module (and hence also as a left $\U^{\wedge_0}$-module).
So it carries a right action of $\U$ commuting with the left $\U^{\wedge_0}$-action:
for a local section $\alpha$ of this sheaf and $\xi\in \g$ we define
$\alpha\xi=\xi\alpha-\xi_{\mathcal{M}}\alpha$, where $\xi_{\mathcal{M}}$ denote
the operator on $\mathcal{M}$ coming from the $G$-action. In particular,
the center of $\U$ acts by endomorphisms of the weakly $G$-equivariant
left $\U$-module $\mathcal{M}$. The right action of the center
also commutes with the left action of $\mathcal{D}^{\wedge_0}$. This is because
the coherent sheaf $\mathcal{M}$ on $\tilde{\g}^{(1)}_\Ring$ is locally free
and $\Dcal^{\wedge_0}$ and $\U^{\wedge_0}$ coincide over $\g^{*(1),reg}$.
Let $\mathcal{M}^{-\rho}$ denote the maximal subsheaf of $\mathcal{M}$, where the action of
$\U^G\cong \F[\h^*/(W,\cdot)]$ factors through $\F[\h^*/(W,\cdot)]^{\wedge_{-\rho}}$.

We are going to show that there is a $G$-equivariant $\mathcal{D}^{\wedge_0}$-module isomorphism
\begin{equation}\label{eq:splitting_D_iso}\mathcal{M}^{-\rho}\cong \mathcal{O}_{\mathcal{B}}(\rho)\otimes_{\mathcal{O}_{\mathcal{B}}}\mathcal{D}^{\wedge_{-\rho}}.
\end{equation}
Consider the sheaf $\U^0$ on $\mathcal{B}$ as in \cite[Section 2]{BB}, it is generated
by  $\g$ (as a Lie algebra) and $\mathcal{O}_{\mathcal{B}}$ (as a sheaf of algebras)
with a cross-relation $[\xi,f]=\xi.f$ for the usual $\g$-action on $\mathcal{O}_{\mathcal{B}}$.
It maps to $\mathcal{D}^{\wedge_{\mu}}$ for all $\mu$. The map is not an epimorphism however
the $\Ring$-submodule generated by the image coincides with $\mathcal{D}^{\wedge_{\mu}}$.
This is because the composition of the map $\U^0\rightarrow \mathcal{D}^{\wedge_{\mu}}$
with the projection
$\mathcal{D}^{\wedge_{\mu}}\twoheadrightarrow
\mathcal{D}^{\wedge_{\mu}}\otimes_{\Ring}\F=D^{\mu}_{\mathcal{B}}$ is surjective,
$\mathcal{D}^{\wedge_{\mu}}$ is a coherent sheaf over $\tilde{\g}^{(1)}_\Ring$,
and the support of every coherent sheaf on $\tilde{\g}^{(1)}_\Ring$ intersects the zero fiber
$\tilde{\Nilp}^{(1)}$. As in \cite[(ii)]{BB},
$\mathcal{M}$ is filtered by $G$-equivariant $\Dcal^{\wedge_0}$-modules of the form $\Dcal^{\wedge_0}\otimes \mathcal{O}(\nu)$,
where $\nu$ is a weight of $L(\rho)$. The latter tensor product is a right $\Dcal^{\wedge_{-\nu}}$-module
thanks to (\ref{eq:D_alg_iso}).  Therefore the action of $\U^G\cong \F[\h^*/(W,\cdot)]$
on $\Dcal^{\wedge_0}\otimes \mathcal{O}(\nu)$ from the right is via $\F[\h^*/(W,\cdot)]^{\wedge_{-\nu}}$.
The condition that $\F[\h^*/(W,\cdot)]^{\wedge_{-\nu}}=\F[\h^*/(W,\cdot)]^{\wedge_{-\rho}}$
means that $\nu+\rho\in p\Lambda$. We claim that this means that $\nu=-\rho$. Indeed,
let $w\in W$ be such that $w(\nu+\rho)$ is dominant (in $p\Lambda$).
We have $0\leqslant w(\nu+\rho)\leqslant 2\rho$,
where the first inequality is an equality if and only if $\nu=-\rho$. In particular,
for every simple coroot $\alpha^\vee$, we have $0\leqslant \langle w(\nu+\rho),\alpha^\vee\rangle\leqslant 2$. We conclude that $w(\nu+\rho)$ only holds if $\nu=-\rho$.

Thanks to this filtration on $\mathcal{M}$, the subsheaf $\mathcal{M}^{-\rho}$ is a direct summand of
$\mathcal{M}$.
The weight $-\rho$ occurs with multiplicity $1$, so
$$\mathcal{M}^{-\rho}\cong \mathcal{D}^{\wedge_0}\otimes_{\mathcal{O}_{\mathcal{B}}}
\mathcal{O}_{\mathcal{B}}(\rho).$$
The right hand side is identified with the right hand side of
(\ref{eq:splitting_D_iso}) as a strongly $G$-equivariant $\mathcal{D}^{\wedge_0}-\mathcal{D}^{\wedge_{-\rho}}$-bimodule, thanks to
(\ref{eq:D_alg_iso}).

We conclude that $$\mathsf{P}^{0,-\rho}\cong
\Gamma(\mathcal{O}_{\mathcal{B}}(\rho)\otimes_{\mathcal{O}_{\mathcal{B}}}\mathcal{D}^{\wedge_{-\rho}}),$$
an isomorphism of strongly $G$-equivariant $\U^{\wedge_0}$-$\U^{\wedge_{-\rho}}$-modules.
A completely similar argument shows that
$$\mathsf{P}^{-\rho,0}\cong \Gamma(\mathcal{D}^{\wedge_{-\rho}}\otimes_{\mathcal{O}_{\mathcal{B}}}\mathcal{O}_{\mathcal{B}}(-\rho)).$$
From the construction of the splitting bundle $\Ecal$, see Definition \ref{defi:Ecal},
we observe that over $\g_\h^{*(1),reg}$ (where $\U^{\wedge_0}$ and $\Dcal^{\wedge_0}$
coincide), the bundle $\Ecal$ coincides with $\mathsf{P}^{0,-\rho}\widehat{\otimes}_{\U}\mathsf{P}^{-\rho,0}$. This finishes the proof.

\end{proof}


We claim that the third functor in (\ref{eq:functor_composition}) is a fully faithful embedding. For the same reason as in the proof of \cite[Proposition 3.7]{BR}, the restriction functor
$\Coh^G(\St^{(1),reg})\rightarrow \operatorname{Rep}(\mathfrak{I})$ is a category equivalence.
The third functor in (\ref{eq:functor_composition}) is obtained from this equivalence by changing
the base to $\Spec(\Ring)$. Arguing as in the proof of \cite[Proposition 3.7]{BR} (or Lemma \ref{Lem:full_embedding_geometric}),
we see that it is fully faithful (it is also an equivalence but we will not need that). By \cite[Proposition 3.7]{BR}, we see that (\ref{eq:functor_HC_slice})
is fully faithful on $\HC_{diag}^G(\U^{\wedge_0})$.

Here is the main result of this section, it is a slightly modified version of
\cite[Theorem 6.3]{BR}.

\begin{Prop}\label{Prop:BR}
There is a full $\Ring$-bilinear monoidal
embedding $\SBim^\wedge\hookrightarrow \HC^G_{diag}(\U^{\wedge_0})$,
sending $B_s^{AS}$ to $B_s^{HC}$ for all simple affine reflections
$s$ and $\Delta^{AS}_x$ to $\Delta^{HC}_x$ for all $x\in \Lambda/\Lambda_r$.
\end{Prop}
\begin{proof}
By Lemma \ref{Lem:full_embedding_geometric} we have a full embedding
$\SBim^\wedge\hookrightarrow \operatorname{Rep}(\mathfrak{J}^\wedge)$
and hence $\SBim^\wedge\hookrightarrow \operatorname{Rep}(\mathfrak{I}^\wedge)$.
By results recalled in this section right after (\ref{eq:functor_HC_slice}), we also have a full embedding
$\HC_{diag}^G(\U^{\wedge_0})\hookrightarrow \operatorname{Rep}(\mathfrak{I}^\wedge)$.
By \cite[Proposition 6.6]{BR}, the images of $B_s^{AS}$ and $B_s^{HC}$
in $\operatorname{Rep}(\mathfrak{I}^\wedge)$  coincides for all simple affine reflections $s$.
By \cite[Lemma 6.8]{BR}, the images of $\Delta_x^{AS},\Delta_x^{HC}$ coincide for all
$x\in \Lambda/\Lambda_r$. This finishes the proof.
\end{proof}


\section{Further study of modular Harish-Chandra bimodules}
In this section we establish further properties of various categories of Harish-Chandra bimodules.
\subsection{Derived localization for Harish-Chandra bimodules}
The goal of this section is to prove the following theorem (Theorem \ref{Thm:localization}
from the introduction).

\begin{Thm}\label{Thm:derived_loc_HC} We have a monoidal exact $\Ring$-bilinear equivalence of
triangulated categories $$D^b(\HC^G(\U^{\wedge_0}))\xrightarrow{\sim}
D^b(\Coh^{G^{(1)}}(\St^{(1)}_\Ring)).$$
\end{Thm}
\begin{proof}
Recall equivalence (\ref{eq:der_loc_St}), in the notation of Section
\ref{SS_BR_HC} it is
\begin{equation}\label{eq:der_loc_St1}
D^b(\overline{\HC}^G(\U^{\wedge_0}))\xrightarrow{\sim}
D^b(\Coh^{G}(\St^{(1)}_\Ring)).
\end{equation}
Also recall the second equivalence in (\ref{eq:derived_equiv_bimod}):
\begin{equation}\label{eq:derived_equiv_bimod1}
D^b(\Coh^{G}\St_{\Ring}^{(1)})\xrightarrow{\sim}
D^b(\Acal^{diag}_{\Ring}\operatorname{-mod}^{G}).
\end{equation}
We will show that
\begin{itemize}
\item[(I)] the composition of (\ref{eq:derived_equiv_bimod1}) and (\ref{eq:der_loc_St1}),
\begin{equation}\label{eq:interm_equiv_local}
D^b(\overline{\HC}^G(\U^{\wedge_0}))\xrightarrow{\sim}
D^b(\Acal^{diag}_{\Ring}\operatorname{-mod}^{G})
\end{equation}
is t-exact,
\item[(II)] and the resulting equivalence
$\overline{\HC}^G(\U^{\wedge_0})\xrightarrow{\sim}\Acal^{diag}_{\Ring}\operatorname{-mod}^{G}$
restricts to
\begin{equation}\label{eq:bimod_equiv_main}
\HC^G(\U^{\wedge_0})\xrightarrow{\sim}\Acal^{diag}_{\Ring}\operatorname{-mod}^{G^{(1)}}.
\end{equation}
\end{itemize}
(II) will yield an equivalence
$D^b(\HC^G(\U^{\wedge_0}))\xrightarrow{\sim}D^b(\Acal^{diag}_{\Ring}\operatorname{-mod}^{G^{(1)}})$.
Combining this with the first equivalence in (\ref{eq:derived_equiv_bimod}) will give an equivalence
\begin{equation}\label{eq:final_loc_equiv} D^b(\HC^G(\U^{\wedge_0}))\xrightarrow{\sim}
D^b(\Coh^{G^{(1)}}(\St^{(1)}_\Ring)).
\end{equation}
Since all intermediate equivalences are $\Ring$-bilinear and monoidal, so is (\ref{eq:final_loc_equiv}).
So (I) and (II) finish the proof of the theorem.

{\it Proof of (I)}: We note that the support in
\begin{equation}\label{eq:small_neighborhood}
\g^{(1)*}\times_{\h^{*(1)}/W}\operatorname{Spec}(\Ring^W)
\end{equation} (a ``small neighborhood''
of the nilpotent cone)
of any object in $\overline{\HC}^G(\U^{\wedge_0})$ intersects $\Spec(\F[[\g^{(1)*}]])$ because
$G^{(1)}$ has only finitely many orbits in the nilpotent cone and zero is the only closed
orbit. The same is true for the objects in $\Acal^{diag}_{\Ring}\operatorname{-mod}^G$.
So it is enough to verify the t-exactness claim after changing the base from
(\ref{eq:small_neighborhood})
 to $\Spec(\F[[\g^{(1)*}]])$. We decorate the corresponding objects
with superscript $\bullet^{\wedge_\g}$, for example
$$\St^{(1)\wedge_\g}:=\St^{(1)}\times_{\g^{*(1)}}\Spec(\F[[\g^{(1)*}]]), \Dcal^{\wedge_0,\g}:=
\Dcal^{\wedge_0}\otimes_{\F[\g^{(1)*}]^\wedge}\F[[\g^{(1)*}]].$$
Recall that, by Remark \ref{Rem:splitting_coincidence},  $\Ecal^{diag,\wedge_\g}$ coincides
with the restriction of $\Ecal'\boxtimes \Ecal'^*$ to $\St^{\wedge_\g}$. By Lemma
\ref{Lem:same_summands}, $\Tcal_\h^{\wedge_\g}$ has the same indecomposable summands as $\Ecal'$.
The claim that the base change of (\ref{eq:interm_equiv_local}) to $\operatorname{Spec}(\F[[\g^{(1)*}]])$
is t-exact follows. So (\ref{eq:interm_equiv_local}) itself is t-exact.

{\it Proof of (II)}: First, we  observe that in both categories $\Acal^{diag}_{\Ring}\operatorname{-mod}^{G^{(1)}},
\HC^G(\U^{\wedge_0})$  every object is a quotient
of a module that is flat over $\F[\g^{*(1)}]$.

In the category $\Acal^{diag}_{\Ring}\operatorname{-mod}^{G^{(1)}}$
we can take objects of the form $\Acal^{diag}_{\Ring}\otimes V'$ for a finite dimensional rational
representation $V'$ of $G^{(1)}$. The reason they are flat is that $\Acal_\h$ is flat over $\F[\g^{*(1)}]$,
see \cite[Section 2.5]{BM}, hence $\Acal^{diag}_{\Ring}$ is flat over $\F[\g^{*(1)}]$.

In the category $\HC^G(\U^{\wedge_0})$ we can take the objects of the form
$\operatorname{pr}_0(V\otimes \U^{\wedge_0})$ for a finite dimensional rational representation
$V$ of $G$.

An object in $\Acal^{diag}_{\Ring}\operatorname{-mod}^{G}$ that is flat over
$\F[\g^{*(1)}]$ lies in $\Acal^{diag}_{\Ring}\operatorname{-mod}^{G^{(1)}}$ if
and only if the action of $G$ on the restriction of this object to an arbitrary nonempty open subscheme
in $\g^{*(1)}_\Ring$ factors through $G^{(1)}$. The similar claim holds for
$\HC^G(\U^{\wedge_0})\subset \overline{\HC}^G(\U^{\wedge_0})$: it consists of strongly
$G$-equivariant objects.
It follows that it is enough to show the claim of (II) after restricting to
a nonempty open subscheme in $\g^{*(1)}_{\Ring}$, for example, to $\g^{*(1),reg}_{\Ring}$.

Over the latter locus equivalences (\ref{eq:der_loc_St1}) and
(\ref{eq:derived_equiv_bimod1}) are t-exact. For $\lambda,\mu$ in the weight lattice of $G$
(or of its cover)
it makes sense to talk about strongly $G$-equivariant objects in
$\Coh^G(\Dcal^{\wedge_{\mu}}\boxtimes \left(\Dcal^{\wedge_{\lambda}}\right)^{opp}|_{\St^{reg}})$. 

So we need to prove that for $\mathcal{F}\in \Coh^G(\St^{(1),reg}_{\Ring})$ the following two conditions
are equivalent:
\begin{itemize}
\item[(i)] The action of $G$ on $\mathcal{F}$ factors through $G^{(1)}$,
\item[(ii)] and $\Ecal^{diag}\otimes \mathcal{F}$ is strongly $G$-equivariant.
\end{itemize}
Tensoring an object in $\Coh^{G^{(1)}}(\St^{(1),reg}_{\Ring})$ with a strongly equivariant object in
$\Coh^G(\Dcal^{\wedge_{0}}\boxtimes \left(\Dcal^{\wedge_{0}}\right)^{opp}|_{\St^{reg}})$ gives
a strongly equivariant object in $\Coh^G(\Dcal^{\wedge_{0}}\boxtimes \left(\Dcal^{\wedge_{0}}\right)^{opp}|_{\St^{reg}})$.
Note also that if two objects in $\Coh^G(\Dcal^{\wedge_{0}}\boxtimes \left(\Dcal^{\wedge_{0}}\right)^{opp}|_{\St^{reg}})$ are strongly equivariant, then  the $G$-action on their Hom
(over $\Dcal^{\wedge_{0}}\boxtimes \left(\Dcal^{\wedge_{0}}\right)^{opp}|_{\St^{reg}}$) factors through
$G^{(1)}$. So the equivalence of (i) and (ii) will follow once we check that $\Ecal^{diag}$ is strongly
equivariant. Note that twist with line bundles on $\mathcal{B}$ preserves the
strong equivariance. This is because
$$\Dcal^{\wedge_{\mu}}\otimes_{\Ocal_\B}\Ocal_{\B}(\mu-\lambda)\in
\Coh^G(\Dcal^{\wedge_{\mu}}\boxtimes \left(\Dcal^{\wedge_{\lambda}}\right)^{opp})$$
is strongly equivariant.

On the other hand, the regular $\U^{\wedge_{-\rho}}$-bimodule is strongly equivariant hence so is its pullback to $\St_{\Ring}^{(1)}$. From the construction of
$\Ecal^{diag}$ in the proof of (I), it follows that
$\Ecal^{diag}$ satisfies the HC condition, which finishes the proof of (II).
\end{proof}

\begin{Rem}\label{Rem:derived_loc_HC}
We will need different versions of the equivalences in Theorem \ref{Thm:derived_loc_HC}.
Consider the full subcategory $\HC^G(\U)_0$ of all objects in $\HC^G(\U^{\wedge_0})$
such that the action of $\Ring$ on the right factors through the residue field, equivalently, the action
of $\U^{\wedge_0}$ factors through $\U_0$. In other words, the objects in $\HC^G(\U)_0$
are exactly the finitely generated strongly $G$-equivariant $\U$-$\U_0$-bimodules.  Since (\ref{eq:bimod_equiv_main})
is $\Ring$-bilinear, it restricts to an equivalence
\begin{equation}\label{eq:bimod_equiv_abelian_specialized}
\HC^G(\U)_0\xrightarrow{\sim} \Acal_\Ring\otimes_{\F[\g^{(1)*}]}\Acal^{opp}\operatorname{-mod}^{G^{(1)}}
\end{equation}
Note that this is (\ref{eq:HC_NCS_equiv}) from Theorem \ref{Thm:HCO_NCS}.

Combining (\ref{eq:bimod_equiv_abelian_specialized}) with (\ref{eq:derived_equiv_bimod_specialized}), we get
\begin{equation}\label{eq:HC_derived_specialized}
D^b(\HC^G(\U)_0)\xrightarrow{\sim} D^b(\Coh^{G^{(1)}}\St^{(1)}).
\end{equation}
Now let $\HC^G(\U_0)$ denote the category of HC $\U_0$-bimodules. Similarly to
(\ref{eq:bimod_equiv_abelian_specialized}) we get an equivalence
\begin{equation}\label{eq:bimod_equiv_abelian_specialized2}
\HC^G(\U_0)\xrightarrow{\sim} \Acal\otimes_{\F[\g^{(1)*}]}\Acal^{opp}\operatorname{-mod}^{G^{(1)}}.
\end{equation}
It is monoidal because (\ref{eq:bimod_equiv_main}) is. And (\ref{eq:bimod_equiv_abelian_specialized})
becomes an equivalence of bimodule categories over the equivalent monoidal categories
from (\ref{eq:bimod_equiv_main}) acting on the left and (\ref{eq:bimod_equiv_abelian_specialized2})
acting on the right.
\end{Rem}

\subsection{Grothendieck group}
The goal of this section is to use Theorem \ref{Thm:derived_loc_HC} to study the $K_0(\HC^G(\U)_0)$.

We have a classical action of $W^{ea}$ on $K_0(\HC^G(\U)_0)$ coming from the reflection functors: tensoring with $B_s^{HC}$ (on the left)
corresponds to the operator $1+s$ for all simple affine reflections,
and tensoring with $\Delta^{HC}_x$ corresponds to $x$ for all $x\in \Lambda/\Lambda_r$.

Our goal is to prove the following result.

\begin{Prop}\label{Prop:HC_K_group}
We have a $W^{ea}$-equivariant isomorphism $\Z W^{ea}\xrightarrow{\sim} K_0(\HC^G(\U)_0)$ that maps
$1\in \Z W^{ea}$ to the class of $\U_0$.
\end{Prop}
\begin{proof}
Thanks to (\ref{eq:HC_derived_specialized}), we have
$$K_0(\HC^G(\U)_0)\xrightarrow{\sim} K_0^{G^{(1)}}(\St^{(1)}).$$
According to \cite[Theorem 7.2.2]{CG}, we have an isomorphism
\begin{equation}\label{eq:K_theory_iso}\Z W^{ea}\xrightarrow{\sim}K_0^{G^{(1)}}(\St^{(1)}).
\end{equation}

The construction
is as follows. Take an element $wt_\lambda\in W^{ea}$ with $w\in W,\lambda\in \Lambda$.
Consider the graph of $W$ in $\St_\h^{(1),rs}$, where the superscript ``rs''
means ``regular semisimple''. Note that
$$\St_\h^{(1),rs}\cong G^{(1)}\times^{N_{G^{(1)}}(T^{(1)})}(\h^{*(1),reg}\times_{\h^{*(1),reg}/W}\h^{*(1),reg})$$
and the graph in question is
$$\Lambda^{reg}_w:=G^{(1)}\times^{N_{G^{(1)}}(T^{(1)})}\{(wx,x), x\in\h^{*(1),reg}\}$$
Let $\Lambda_w$ be the scheme theoretic intersection of the closure of $\Lambda^{reg}_w$
in $\St_{\h}^{(1)}$ with its reduced subscheme $\St^{(1)}$. We write $\mathcal{O}_{\Lambda_w}(\lambda)$
for the pullback of $\mathcal{O}_{\tilde{\Nilp}^{(1)}}(\lambda)$ under the projection to
the second factor. Then the isomorphism (\ref{eq:K_theory_iso}) sends
$wt_\lambda$ to the class of $\mathcal{O}_{\Lambda_w}(\lambda)$. We note that
$K_0^{G^{(1)}}(\St^{(1)})$ carries an algebra structure by convolution
and  (\ref{eq:K_theory_iso}) is an algebra isomorphism. We also remark that,
while \cite[Theorem 7.2.2]{CG} is stated over $\mathbb{C}$, the proof carries to our setting
verbatim.

It remains to prove that this
identification is $W^{ea}$-equivariant. Let $\St_{\Ring}^{(1),rs}$ denote the intersection
of $\St_\h^{(1),rs}$ with $\St_\Ring^{(1)}$.
We have a natural functor from $\SBim^\wedge$ to
\begin{equation}\label{eq:regular_cats}
\Coh^{G^{(1)}}(\St_{\Ring}^{(1),rs})\cong \Coh^{N_{G^{(1)}}(T^{(1)})}(\operatorname{Spec}(\Ring)^{reg}
\times_{\operatorname{Spec}(\Ring^W)}\operatorname{Spec}(\Ring)^{reg})
\end{equation} via the localization
to the regular locus.  Note that, by the construction of the objects $\Delta_x^{AS},B_s^{AS}\in \SBim^{\wedge}$, their images in the right hand side of (\ref{eq:regular_cats}) are
the graph of $x$ and the extension of the graph of $1$ by the graph of $s$, respectively
(for all simple affine reflections $s$ and all $x\in \Lambda/\Lambda_r$).
On the other hand, we also have a functor from $\HC^G(\U^{\wedge_0})$ to $\Coh^G(\St_{\Ring}^{(1),rs})$,
thanks to Theorem \ref{Thm:derived_loc_HC} or, equivalently, the construction in
Section \ref{SS_BR_HC}. This functor intertwines the actions on $K_0(\HC^G(\U)_0)\cong K_0(\Coh^{G^{(1)}}(\St^{(1)}))$.
Thanks to the proof of Proposition \ref{Prop:BR}, the images of $B_s^{HC}, \Delta_x^{HC}$ in
$\Coh^G(\St_{\Ring}^{(1),rs})$ coincide with those of $B_s^{AS},\Delta_x^{AS}$.
Together with the construction of the previous paragraph,
this implies that
the identification $\Z W^{ea}\xrightarrow{\sim} K_0(\Coh^{G^{(1)}}(\St^{(1)}))$
is $W^{ea}$-equivariant.
\end{proof}

\subsection{Duality functor}
Proposition \ref{Prop:HC_K_group} has a useful application to the study of the duality functor
on $\HC^G(\U)_0$. We now recall how this functor is defined.
Similarly to Section \ref{SS_BR_HC}, $\HC^G(\U)_0$ is a direct summand in the category $\U_0^{opp}\operatorname{-mod}^G$ consisting of all bimodules with generalized central character $0$
on the left.   Similarly to $\U_0\operatorname{-mod}^G$, we can consider the
category $\U_0\operatorname{-mod}^G$.

Note that $\U_0$ is a Gorenstein algebra because its associated graded, $\F[\Nilp]$, is.
So the functor $R\Hom_{\U_0^{opp}}(\bullet,\U_0)$ is an equivalence
$D^b(\U_0^{opp}\operatorname{-mod}^G)\xrightarrow{\sim} D^b(\U_0\operatorname{-mod}^G)^{opp}$.
Let $\theta$ denote the Cartan involution of $G$: on the Lie algebra it sends the Cartan generators
$e_i$ to $f_i$ and vice versa. It gives rise to an equivalence $\U_0\operatorname{-mod}^G
\rightarrow \U_0^{opp}\operatorname{-mod}^G$ that twists the action of $\g$ by the antiautomorphism
$x\mapsto -\theta(x)$ (so that the right action of $\U_0$ gives rise to a left action of $\U_0$
because the action of $-\theta$ on the center is trivial) and twists the action of $G$ by $\theta$.
Denote this equivalence $\U_0^{opp}\operatorname{-mod}^G\xrightarrow{\sim} \U_0\operatorname{-mod}^G$
by $M\mapsto \,^\theta\! M$.

Consider the contravariant auto-equivalence  $\Dual:=\,^\theta\!R\Hom_{\U_0}(\bullet,\U_0)$ of $D^b(\U_0^{opp}\operatorname{-mod}^G)$.

On the other hand, we have an action of the extended affine braid
group $\operatorname{Br}^{ea}$ on $D^b(\HC^G(\U)_0)$ by wall-crossing functors: the length zero elements $x$
acts by $\Delta^{HC}_x\otimes_{\U^{\wedge_0}}\bullet$, while for a simple affine reflection $s$
the corresponding generator $T_s$ acts by tensoring with the cone of $\operatorname{Id}\rightarrow \Theta_s$, where $\Theta_s$ is the classical reflection functor.

The following are basic properties of this equivalence.
They are  standard.

\begin{Lem}\label{Lem:dual_basic}
The following claims are true:
\begin{enumerate}
\item The equivalence $\Dual$ restricts to a contravariant auto-equivalence of
$D^b(\HC^G(\U)_0)$. Moreover, $\Dual^2\cong \operatorname{id}$.
\item For $M\in D^b(\U_0^{opp}\operatorname{-mod}^G)$ and a finite dimensional rational
representation $V$ of $G$ we have $\Dual(V\otimes M)\xrightarrow{\sim} V^\vee \otimes \Dual(M)$
for $V^\vee:=\,^\theta V^*$.
\item For $x\in W^{ea}$ we have $T_x\circ \Dual\cong \Dual\circ T_{x^{-1}}^{-1}$. Here we write
$T_x$ for the wall-crossing equivalence of $D^b(\HC^G(\U)_0)$ corresponding to $x\in W^{ea}$.
\end{enumerate}
\end{Lem}

\begin{Cor}\label{Cor:dual_K_0}
The functor $\Dual$ gives the identity on $K_0(\HC^G(\U_0))$.
\end{Cor}
\begin{proof}
We have $\Dual(\U_0)\cong \U_0$. Thanks to (3) of Lemma \ref{Lem:dual_basic} we see that
$\Dual$ acts on $K_0(\HC^G(\U)_0)$ by a $W^{ea}$-equivariant map. Now we are done by
Proposition \ref{Prop:HC_K_group}.
\end{proof}

\begin{Rem}\label{Rem:deformed_duality}
Consider $D^b(\HC^G(\U^{\wedge_0}))$ as a direct summand
of $D^b(\U^{\wedge_0,opp}\operatorname{-mod}^G)$. This allows us to define
the contravariant auto-equivalence $\Dual_{\Ring}$ of $D^b(\HC^G(\U^{\wedge_0}))$
by $$\Dual_{\Ring}:=\,^\theta\! R\Hom_{\U^{\wedge_0}}(\bullet,\U^{\wedge_0})$$ For $M\in \HC^G(\U)_0$ we have
$\Dual_{\Ring}(M)=\Dual(M)[\dim \h]$ because $\U^{\wedge_0}$ is flat over $\Ring$. We also note that $\Dual_{\Ring}$
satisfies the direct analog   of Lemma
\ref{Lem:dual_basic}.
\end{Rem}

\subsection{HC-tilting bimodules}
Here we define a full subcategory of {\it HC-tilting} bimodules in
$\HC^G(\U^{\wedge_0})$ and in $\HC^G(\U)_0$.

\begin{defi}\label{defi:hilting}
An object in $\HC^G(\U^{\wedge_0})$ is called {\it HC-tilting} if it is a direct summand
in $T\otimes \U^{\wedge_0}$ for a tilting $G$-module $T$.
\end{defi}

One can define HC-tilting objects in $\HC^G(\U)_0$ similarly. Let $\Hilt^G(\U^{\wedge_0})$
and $\Hilt^G(\U)_0$ denote the full subcategories of HC-tilting bimodules in the respective categories.

The following lemma describes basic properties of HC-tilting bimodules. The proofs are standard
and so are omitted.

\begin{Lem}\label{Lem:HC-tilting_basic}
The following claims are true:
\begin{enumerate}
\item $\Hilt^G(\U^{\wedge_0})$ is a Karoubian monoidal subcategory of $\HC^G(\U^{\wedge_0})$.
\item The subcategories $\Hilt^G(\U^{\wedge_0}), \Hilt^G(\U)_0$ are stable with respect
to the duality functors $\Dual_\Ring,\Dual$, respectively.
\item We have $B_s^{HC},\Delta_x^{HC}\in \Hilt^G(\U^{\wedge_0})$
for all simple affine reflections $s$ and all $x\in \Lambda/\Lambda_r$.
In particular, the image of the full embedding from Proposition
\ref{Prop:BR} is in $\Hilt^G(\U^{\wedge_0})$.
\end{enumerate}
\end{Lem}

Here is another important property.

\begin{Prop}\label{Prop:Ext_vanishing}
For $\mathcal{B}_1,\mathcal{B}_2\in \Hilt^G(\U)_0$, their higher Ext groups in $\HC^G(\U)_0$
vanish.
\end{Prop}
\begin{proof}
Recall that each object in
$\Hilt^G(\U)_0$ is a direct summand in $T\otimes \U_0$, where $T$ is a tilting
representation of $G$. Also recall that the category $\HC^G(\U)_0$ is a direct summand in
$\U_0^{opp}\operatorname{-mod}^G$. So it is enough to show that for all tilting representations
$T_1,T_2$ of $G$ we have
\begin{equation}\label{eq:HC-tilting_vanishing}
\Ext^i_{\U_0,G}(T_1\otimes \U_0, T_2\otimes \U_0)=0,\forall i>0,
\end{equation}
where the Ext is taken in the category $\U_0^{opp}\operatorname{-mod}^G$.

The duals and tensor products of tiltings are again tilting, see \cite[Proposition 4.19]{Jantzen}
for tensor products. And every tilting
object is Weyl filtered (i.e., filtered by Weyl modules). So it is enough to show that
$$\Ext^i_{\U_0,G}(V\otimes \U_0, \U_0)=0$$
for all Weyl filtered representations $V$ of $G$.

Note that the left hand side is $\Ext^i_G(V,\U_0)$. As a $G$-module, $\U_0$ admits
a resolution (from the left) whose  terms are direct sums several copies
of the $G$-module $\U$. The latter is
costandardly filtered, this follows from \cite[Section 4.21]{Jantzen}.
So $\Ext^i_G(V,\U_0)=0$ for all $i>0$. This implies the claim of the proposition.
\end{proof}

This proposition has a standard corollary.

\begin{Cor}\label{Cor:flat_hom}
The following claims hold:
\begin{enumerate}
\item For $\mathcal{B}_1,\mathcal{B}_2\in \Hilt^G(\U^{\wedge_0})$, we have that
$\Hom(\mathcal{B}_1,\mathcal{B}_2)$ is flat over $\Ring$ and
$\Hom(\mathcal{B}_1\otimes_{\Ring}\F,\mathcal{B}_2\otimes_{\Ring}\F)=
\Hom(\mathcal{B}_1,\mathcal{B}_2)\otimes_\Ring\F$.
\item The functor $\bullet\otimes_{\Ring}\F$ defines a bijection between the indecomposable
objects in $\Hilt^G(\U^{\wedge_0})$ and $\Hilt^G(\U)_0$.
\item For $\mathcal{B}_1,\mathcal{B}_2\in \Hilt^G(\U^{\wedge_0})$, their higher Ext groups
in $\HC(\U^{\wedge_0})$ vanish.
\end{enumerate}
\end{Cor}

\subsection{Perverse bimodules}\label{SS_perverse}
Here we recall a t-structure on $D^b(\HC^G(\U)_0)$ called the {\it perverse t-structure}, compare to
\cite{AB}. For this t-structure we have
\begin{equation}\label{eq:perverse_spec_defn}
\begin{split}
& \,^p\!D^{b,\leqslant 0}(\HC^G(\U)_0)=\{M\in D^b(\HC^G(\U)_0)| \dim \operatorname{Supp}H^i(M)\leqslant
\dim \Nilp-2i\},\\
& \,^p\!D^{b,\geqslant 0}(\HC^G(\U)_0)=\Dual(\,^p\!D^{b,\leqslant 0}(\HC^G(\U)_0)).
\end{split}
\end{equation}

The proof that this is indeed a t-structure copies that for $\F[\mathcal{N}]$ instead of
$\U_0$ in \cite[Section 3]{AB}.
The heart of this structure will be denoted by $\Perv(\HC^G(\U)_0)$, objects there will be called
{\it perverse bimodules}.

\begin{Ex}\label{Ex:perverse}
Let $V$ be a finite dimensional rational representation. Then $\operatorname{pr}_0(V\otimes \U_0)$
is a perverse bimodule. Indeed, this object lies in $\,^p\!D^{b,\leqslant 0}$ and its dual is
$\operatorname{pr}_0(V^\vee\otimes \U_0)$ by (2) of Lemma \ref{Lem:dual_basic}.
\end{Ex}

\begin{Lem}\label{Lem:finite_length}
All objects in $\Perv(\HC^G(\U)_0)$ have finite length.
\end{Lem}
\begin{proof}
Consider the inclusion $\F[\mathcal{N}^{(1)}]\hookrightarrow \U_0$. It gives rise to
the pullback functor $$D^b(\HC^G(\U)_0)\rightarrow D^b(\F[\mathcal{N}^{(1)}]\operatorname{-mod}^G).$$
Similarly to the proof of \cite[Proposition 6.10]{BL}, this functor is t-exact for the perverse
t-structures. Also the restriction to the hearts is faithful. By \cite[Corollary 4.13]{AB},
every object in $\operatorname{Perv}(\F[\mathcal{N}^{(1)}]\operatorname{-mod}^G)$ has finite length.
Since $\Perv(\HC^G(\U)_0)$ admits a faithful t-exact functor to a category where every object has finite length, every object in $\Perv(\HC^G(\U)_0)$ has finite length too.
\end{proof}

Note that $\Perv(\HC^G(\U)_0)$ is preserved by $\Dual$ and so $\Dual$ is a t-exact contravariant
duality functor of  $\Perv(\HC^G(\U)_0)$.

The following is the main result of this section.

\begin{Prop}\label{Prop:perverse_K0}
The following claims are true:
\begin{enumerate}
\item We have a $W^{ea}$-equivariant isomorphism
$\Z W^{ea}\xrightarrow{\sim} K_0(\Perv(\HC^G(\U)_0))$ of abelian groups
that sends $1\in \Z W^{ea}$
to the class of $\U_0$ (where the $W^{ea}$-action on the target comes from
the reflection functors).
\item For every simple object  $L\in \Perv(\HC^G(\U)_0)$ we have
$\Dual L\cong L$.
\end{enumerate}
\end{Prop}
\begin{proof}
We start by proving (1). We observe that the perverse t-structure is homologically finite:
every object in $D^b(\HC^G(\U)_0)$ has only finitely many nonzero perverse cohomology groups. This follows
from the following two observations:  first, $\,^p\!D^{b,\leqslant 0}\subset D^{b,\leqslant \dim \Nilp/2}$
and, second, the perverse t-structure is self-dual. From the homological finiteness we get an
identification
\begin{equation}\label{eq:K_0_ident}
K_0(\Perv(\HC^G(\U)_0)\xrightarrow{\sim} K_0(D^b(\HC^G(\U)_0))(=K_0(\HC^G(\U)_0)).
\end{equation}
Part (1) follows from (\ref{eq:K_0_ident}) and
Proposition \ref{Prop:HC_K_group}.

Now we prove part (2). By (\ref{eq:K_0_ident}) and Corollary \ref{Cor:dual_K_0},
$\Dual$ acts by $1$ on $K_0(\Perv(\HC^G(\U)_0))$.
Now (2) follows from Lemma \ref{Lem:finite_length}.
\end{proof}

\section{Equivalence with Soergel-type categories $\Ocat$}\label{S_Soergel_O}
\subsection{Categories $\Ocat_{\Ring}$ and $\Ocat$}\label{SS_O}
Recall the completed category $\SBim^{\wedge}$. There is a highest weight category
$\Ocat_{\Ring}$ over $\Ring$ with poset $W^{ea}$ (with respect to the Bruhat
order) whose category of tilting objects is identified with $\SBim^\wedge$. We will
follow the construction from \cite[Section 6]{EL} where the case of a Coxeter group
was treated.

In \cite[Section 6.6.7]{EL} the author and Elias introduced the categories $\,_I\mathcal{O}^-_{\mathcal{R},J}(W^a)$ (in the notation of that paper, note that
there we considered the general Weyl group). We will need the case when
$I=J=\varnothing$. The ring $\mathcal{R}$ in this case is the algebra of formal
power series in the affine Cartan with $p$-adic coefficients. So $\Ring$
is an algebra over $\mathcal{R}$.  Set
\begin{equation}\label{eq:Ocat_wedge_def}
\Ocat_{\Ring}(W^a):=(\,_\varnothing
\mathcal{O}^-_{\mathcal{R},\mathcal{\varnothing}}(W^a))^{opp}\otimes_{\mathcal{R}}\Ring.
\end{equation}
More precisely, $\,_\varnothing
\mathcal{O}^-_{\mathcal{R},\varnothing}(W^a)^{opp}$ is the category of modules
with discrete topology over the inverse limit of $\mathcal{R}$-algebras that are
free finitely generated $\mathcal{R}$-modules (see \cite[Section 6.1.3]{EL}).
We then base change this topological
algebra from $\mathcal{R}$ to $\Ring$ and take the category of modules with discrete
topology for $\Ocat_{\Ring}(W^a)$.

The category $\Ocat_{\Ring}(W^a)$ is an ideal finite highest weight category over $\Ring$ in
the sense of \cite[Section 6.1.3]{EL}. Its highest weight poset is $W^a$ with its
Bruhat order. We write $T_{\Ring}(x),\Delta_{\Ring}(x),\nabla_{\Ring}(x)$
for the indecomposable tilting, standard and costandard objects labelled by $x$,
respectively (by definition, $T_{\Ring}(x)$ is the unique indecomposable tilting
that admits an epimorphism onto $\nabla_{\Ring}(x)$).
By the construction, the category $\Ocat_{\Ring}(W^a)\operatorname{-tilt}$
of tilting objects in $\Ocat_{\Ring}(W^a)$ is identified with $\SBim^\wedge(W^a)$ preserving
the labels.

Thanks to \cite[Lemma 6.30]{EL},
one can inductively construct the costandard objects in $\Ocat_{\Ring}(W^a)$
as follows. The costandard object $\nabla_{\Ring}(1)$ is the indecomposable tilting
$T_{\Ring}(1)$. Let $x\in W^a$ and a simple affine reflection $s$
be such that $sx>x$ in the Bruhat order. Assume we have already constructed the object
$\nabla_{\Ring}(x)$. Let $\Theta_s$ denote the reflection endo-functor
of $\Ocat_{\Ring}$ given by $B_s\in \SBim^\wedge$,  it has a distinguished morphism
from the identity endofunctor. Then thanks to (2) of \cite[Lemma 6.30]{EL} we see that
$\Theta_s\nabla_{\Ring}(x)$ admits a filtration by $\nabla_{\Ring}(x)$ and
$\nabla_{\Ring}(sx)$.
Since $sx>x$ in the highest weight order, it follows $\nabla_{\Ring}(sx)$
is constructed as the cokernel of  $\nabla_{\Ring}(x)\rightarrow
\Theta_s \nabla_{\Ring}(x)$. In particular, this cokernel is
flat over $\Ring$.

The following lemma describes the standard objects in $\Ocat_{\Ring}(W^a)$.
Note that there is a functor morphism $\Theta_s\rightarrow \operatorname{id}$.

\begin{Lem}\label{Lem:Ocat_standards}
We have $\Delta_{\Ring}(1)=T_{\Ring}(1)$. Moreover, if $x\in W^a$
and a simple affine reflection $s$ are such that $sx>x$, then
$\Theta_s\Delta_{\Ring}(x)\rightarrow \Delta_{\Ring}(x)$ is an epimorphism and
$\Delta_{\Ring}(sx)$ is its kernel.
\end{Lem}
\begin{proof}
The claim that $\Delta_{\Ring}(1)=T_{\Ring}(1)$ is standard.

Since $\Theta_s$ is a functor that is isomorphic to its left adjoint and preserves the subcategory
$\Ocat_{\Ring}(W^a)^\nabla$ of costandardly filtered objects, it also preserves the
subcategory  $\Ocat_{\Ring}(W^a)^\Delta$ of standardly filtered objects.
Moreover there is a perfect pairing of Grothendieck groups of exact categories
$K_0(\Ocat_{\Ring}(W^a)^\Delta)\times K_0(\Ocat_{\Ring}(W^a)^\nabla)
\rightarrow \Z$ that sends a pair of objects $(M,N)$ to the rank of
$\Hom(M,N)$. By the construction of the costandard objects recalled above in this section,
$K_0(\Ocat_{\Ring}(W^a)^\nabla)$ is identified with $\Z W^a$. The Hom pairing identifies
$K_0(\Ocat_{\Ring}(W^a)^\Delta)$ with the restricted dual of $\Z W^a$ (all functions $f$
such that $f(x)\neq 0$ only for finitely many elements $x\in W^a$). This module is naturally
identified with $\Z W^a$. Recall that $\Theta_s$ is isomorphic to its left adjoint.
It follows that the identification  $K_0(\Ocat_{\Ring}(W^a)^\Delta)\xrightarrow{\sim}
\Z W^a$ is $W^a$-equivariant, so $\Delta_{\Ring}(x)$ corresponds to $x$
for all $x$. So $\Theta_s \Delta_{\Ring}(x)$ is filtered by $\Delta_{\Ring}(x)$
and $\Delta_{\Ring}(sx)$. Since $sx>x$ in the highest weight order,
we see that
$\Theta_s\Delta_{\Ring}(x)\rightarrow \Delta_{\Ring}(x)$ is an epimorphism and
$\Delta_{\Ring}(sx)$ is its kernel.
\end{proof}

We will need another lemma.

\begin{Lem}\label{Cor:exactness_Theta}
The following two claims hold:
\begin{enumerate}
\item We have an equivalence $K^b(\SBim^\wedge(W^a))\xrightarrow{\sim} D^b(\Ocat_{\Ring}(W^a))$.
\item $\SBim^{\wedge}(W^a)$ acts on $K^b(\SBim^\wedge(W^a))$ by t-exact functors.
\end{enumerate}
\end{Lem}
\begin{proof}
(1) is pretty standard. Recall that $\SBim^\wedge(W^a)\xrightarrow{\sim}
\Ocat_{\Ring}(W^a)\operatorname{-tilt}$. We have a full embedding  $K^b(\Ocat_{\Ring}(W^a)\operatorname{-tilt})\hookrightarrow D^b(\Ocat_{\Ring}(W^a))$ because there are no higher self extensions between tiltings.
Every object in $\Ocat_{\Ring}(W^a)$ admits a finite resolution by standardly filtered objects (from the left).
Every standardly filtered object is in $K^b(\Ocat_{\Ring}(W^a)\operatorname{-tilt})$.
So $K^b(\Ocat_{\Ring}(W^a)\operatorname{-tilt})\hookrightarrow D^b(\Ocat_{\Ring}(W^a))$
is essentially surjective. (1) follows.

Now we prove (2). By Lemma \ref{Lem:Ocat_standards}, the action of $\SBim^\wedge(W^a)$
preserves the subcategory of standardly filtered objects. Since every object in
$\Ocat_{\Ring}(W^a)$ admits a resolution by standardly filtered objects, we see that
$\SBim^{\wedge}(W^a)$ acts by right t-exact functors. But the functors coming from
$\SBim^\wedge(W^a)$ are closed under taking adjoints, so they are also left t-exact.
This finishes the proof of (2).
\end{proof}

Now we explain how to define $\Ocat_{\Ring}$ for the group $W^{ea}$. For this note that
$$\Hom_{\SBim^\wedge}(B_x^{AS}, B_y^{AS})=0\text{ if }xy^{-1}\not \in W^a.$$ So $\SBim^\wedge$ splits
into the direct sum of subcategories indexed by $\Lambda/\Lambda_r=W^{ea}/W^a$. The objects
$\Delta_x^{AS}$ for $x\in \Lambda/\Lambda_r$ are invertible. Moreover, we have the group homomorphism from
$\Lambda/\Lambda_r$ to the group of (isomorphism classes of) invertible objects
in $\SBim^{\wedge}$ given by $x\mapsto \Delta_x^{AS}$.

We set
\begin{equation}\label{eq:O_decomp}\Ocat_{\Ring}:=\bigoplus_{x\in \Lambda/\Lambda_r}\Ocat_{\Ring}(W^a).
\end{equation}
This decomposition gives rise to an identification $\Ocat_{\Ring}\operatorname{-tilt}
\xrightarrow{\sim}\SBim^{\wedge}$. This gives rise an action of $\SBim^\wedge$ on
$\Ocat_{\Ring}$. The objects $\Delta_x^{AS}$ for $x\in \Lambda/\Lambda_r$ act by
category equivalence that permute the summands in (\ref{eq:O_decomp}). The corresponding functor
sends $T_{\Ring}(y),\nabla_{\Ring}(y),\Delta_{\Ring}(y)$ to
$T_{\Ring}(xy),\nabla_{\Ring}(xy),\Delta_{\Ring}(xy)$, respectively. The action of $B_s^{AS}$
(to be denoted by $\Theta_s$) on the costandard and standard objects are as described above.
The direct analog of Lemma \ref{Cor:exactness_Theta} holds.

The category $\Ocat_{\Ring}$ is $\Ring$-linear. So we can consider
its specialization
\begin{equation}\label{eq:Ocat_def}
\Ocat:=\Ocat_{\Ring}\otimes_{\Ring}\F
\end{equation}
This is an
$\F$-linear highest weight category with the same poset $W^{ea}$. We write $$T(x),\Delta(x),\nabla(x),L(x)$$
for the indecomposable tilting, standard, costandard and simple objects labelled by $x\in W^{ea}$.
Note that $\SBim^\wedge$ still acts on $\Ocat$ by exact functors.

\subsection{Duality for categories $\Ocat$}
The goal of this section is to define the  duality functors  for $\Ocat_{\Ring},\Ocat$. We start with the duality
functor for $\SBim^\wedge(=\Ocat_{\Ring}\operatorname{-tilt})$. Recall
the duality functor $\Dual_{\Ring}$ of $D^b(\U^{\wedge_0}\operatorname{-mod})$,
Remark \ref{Rem:deformed_duality}. Note that it preserves $\Hilt^G(\U^{\wedge_0})$,
see (2) of Lemma \ref{Lem:dual_basic} and Remark \ref{Rem:deformed_duality}.
Also recall that $\SBim^\wedge$ fully embeds into $\Hilt^G(\U^{\wedge_0})$,
Proposition \ref{Prop:BR}.

\begin{Lem}\label{Lem:dual_SBim}
The image of $\SBim^{\wedge}$ in $\operatorname{Hilt}^G(\U^{\wedge_0})$
is closed under $\Dual_\Ring$. Moreover, the images of the indecomposable objects $B_x^{AS}$ in
$\operatorname{Hilt}^G(\U^{\wedge_0})$ are self-dual for all $x\in W^{ea}$.
\end{Lem}
\begin{proof}
The image of $B_1$ is $\U^{\wedge_0}$, which is self-dual. Also recall,
Remark \ref{Rem:deformed_duality}, that $\Dual_{\Ring}$ intertwines
the endo-functors of $D^b(\HC^G(\U^{\wedge_0}))$ of the form $\pr_0(V\otimes\bullet)$
where $V$ is a self-dual representation of $G$ (e.g. simple or tilting).
The functors $\Theta_s$ and tensoring with $\Delta^{HC}_x$ are of this form.
It follows that the image of every Bott-Samelson
bimodule $\mathsf{BS}^{AS}_{\underline{w}}$ (associated to a word $\underline{w}=(x,s_{i_1},\ldots,s_{i_k})$) in $\operatorname{Hilt}^G(\U^{\wedge_0})$ is self-dual.
Recall that if $\underline{w}$ is a reduced expression for $w$, then
$B^{AS}_w$ occurs in $\mathsf{BS}^{AS}_{\underline{w}}$ with multiplicity $1$
and all other $B^{AS}_y$ that occur in $\mathsf{BS}^{AS}_{\underline{w}}$ have
$y<w$. Now the claim that the image of $B^{AS}_w$ is self-dual is proved
by induction on $w$ with respect to the Bruhat order (using the observation that
$\Hilt^G(\U^{\wedge_0})$ is a Krull-Schmidt category). In particular, the
image of $\SBim^\wedge$ is self-dual.
\end{proof}

So we get a contravariant $\Ring$-linear self-equivalence
of $\Ocat_{\Ring}\operatorname{-tilt}$ that we again denote
by $\Dual_{\Ring}$. The functor extends to $K^b(\Ocat_{\Ring}\operatorname{-tilt})$ which, by Lemma
\ref{Cor:exactness_Theta}, is identified with $D^b(\Ocat_{\Ring})$.

\begin{Lem}\label{Cor:dual_stand_costand}
The self-equivalence $\Dual_{\Ring}$ of $K^b(\Ocat_{\Ring}\operatorname{-tilt})$
has the following properties.
\begin{enumerate}
\item We have $\Dual_{\Ring}(T_{\Ring}(x))\cong T_{\Ring}(x)$ for all
$x\in W^a$.
\item We have $\Dual_{\Ring}\circ \Theta_s\cong \Theta_s\circ \Dual_{\Ring}$
and $\Dual_{\Ring}(\Delta^{AS}_x\otimes\bullet)\cong \Delta^{AS}_x\otimes \Dual_\Ring(\bullet)$
for all $x\in \Lambda/\Lambda_r$.
\item We have $\Dual_{\Ring}(\nabla_{\Ring}(x))\cong\Delta_{\Ring}(x)$.
\end{enumerate}
\end{Lem}
\begin{proof}
(1) follows directly from Lemma \ref{Lem:dual_SBim}. (2) follows from the construction
of $\Dual_{\Ring}$, compare to the proof of Lemma \ref{Lem:dual_SBim}. We prove (3)
by induction on $x$ with respect to the Bruhat order. The base, $x\in \Lambda/\Lambda_r$, follows from
(1). To establish the induction step, note that
\begin{equation}\label{eq:rk1_Hom}\Hom_{\Ocat_{\Ring}}(\nabla_\Ring(x),\Theta_s \nabla_{\Ring}(x))\xrightarrow{\sim}
\Ring.
\end{equation}
For an invertible element  $\varphi\in \Ring$ viewed as a homomorphism, $\varphi$ is a monomorphism and
the cokernel of $\varphi$ is $\nabla_{\Ring}(sx)$, this follows from the reminder
in Section \ref{SS_O}. The functor $\Dual$ sends
this cokernel to the kernel of an arbitrary generator of the $\Ring$-module
$\Hom_{\Ocat_{\Ring}}(\Theta_s\Delta_{\Ring}(x), \Delta_{\Ring}(x))$.
The latter coincides with $\Delta_{\Ring}(sx)$, see
Lemma \ref{Lem:Ocat_standards}. This finishes the proof of (3).
\end{proof}

We proceed to  the duality functor for the category $\Ocat$.
Note that since we have an $\Ring$-linear
isomorphism $\Dual_{\Ring}(T_{\Ring}(x))\cong T_{\Ring}(x)$,
we get $\Dual_{\Ring}(T(x))\cong T(x)[-\dim \h]$. So we have
a duality functor $\Dual:=\Dual_{\Ring}[\dim \h]$ of $\Ocat\operatorname{-tilt}$.
We then extend $\Dual$ to $K^b(\Ocat\operatorname{-tilt})\cong D^b(\Ocat)$.

\begin{Lem}\label{Lem:dual_O}
The following claims are true:
\begin{enumerate}
\item We have $\Theta_s\circ \Dual\cong \Dual\circ\Theta_s$ and $\Delta^{AS}_{x}\otimes\Dual(\bullet)
\cong \Dual(\Delta^{AS}_x\otimes\bullet)$ for all simple affine
reflections $s$ and all $x\in \Lambda/\Lambda_r$.
\item We have $\Dual(\Delta(x))\cong \nabla(x)$ for all $x\in W^a$.
\item The functor $\Dual$ is t-exact.
\item We have $\Dual(L(x))\cong L(x)$.
\end{enumerate}
\end{Lem}
\begin{proof}
(1) and (2) follow from (2) and (3) of Lemma \ref{Cor:dual_stand_costand}. Every object in
$\Ocat$ admits a resolution $0\rightarrow M_k\rightarrow\ldots\rightarrow M_0$, where
all $M_i$ are standardly filtered and also a resolution $N_0\rightarrow \ldots\rightarrow N_\ell\rightarrow 0$, where all $N_j$ are costandardly filtered. Now (3) follows from (2) combined with
$\Dual^2\cong \operatorname{id}$. Since $L(x)$ is the image of the unique (up to scaling) homomorphism
$\Delta(x)\rightarrow \nabla(x)$, (4) follows from (2) and (3).
\end{proof}

\subsection{Equivalence $\Ocat\xrightarrow{\sim} \Perv(\HC^G(\U)_0)$}
This is the main section of the paper. Here we construct an equivalence
$\Ocat\xrightarrow{\sim} \Perv(\HC^G(\U)_0)$. Let us explain the four main
steps of the construction and the proof.

\begin{itemize}
\item[(i)] We use $\SBim^\wedge\hookrightarrow \Hilt^G(\U^{\wedge_0})$
to  produce a full  $\SBim^\wedge$-equivariant
embedding $\Ocat\operatorname{-tilt}\hookrightarrow \Hilt^G(\U)_0$ intertwining the duality
functors.
\item[(ii)] We show that $\Ocat\operatorname{-tilt}\hookrightarrow \Hilt^G(\U)_0$
extends to an $\SBim^\wedge$-equivariant full embedding
$D^b(\Ocat)\hookrightarrow D^b(\HC^G(\U)_0)$ intertwining the duality
functors.
\item[(iii)] We know that the duality functor for $\Perv(\HC^G(\U)_0)$
fixes all simples. We use this and the fact that $D^b(\Ocat)\hookrightarrow D^b(\HC^G(\U)_0)$
is a full embedding intertwining the duality functors to show that this full embedding
is t-exact and sends simples to simples (on the right we consider the perverse t-structure).
\item[(iv)] We show that the resulting full embedding $\Ocat\hookrightarrow \Perv(\HC^G(\U)_0)$
is an equivalence by looking at the induced map between the Grothendieck groups.
\end{itemize}

We start with (i). We can view $\Ocat\operatorname{-tilt}$ as a full subcategory
in $K^b(\Ocat_{\Ring}\operatorname{-tilt})$: we send $T(x)$ to the Koszul
complex for the action of $\h$ on $T_{\Ring}(x)$. Similarly,
$\Hilt^G(\U)_0$ embeds into $K^b(\Hilt^G(\U^{\wedge_0}))$. Recall
that $\Ocat_{\Ring}\operatorname{-tilt}\cong \SBim^\wedge$. Clearly,
the full embedding $K^b(\Ocat_{\Ring}\operatorname{-tilt})\hookrightarrow
K^b(\Hilt^G(\U^{\wedge_0}))$ (which is $\SBim^\wedge$-equivariant and intertwines
the duality functors) induces a full embedding $\Ocat\operatorname{-tilt}
\hookrightarrow \Hilt^G(\U)_0$. By construction, this embedding is
$\SBim^\wedge$-equivariant and intertwines the duality functors. This completes (i)
above.

(ii) is the following lemma.

\begin{Lem}\label{Lem:ext_derived_cats}
The full embedding $\Ocat\operatorname{-tilt}\hookrightarrow \Hilt^G(\U)_0$
uniquely extends to an exact functor $\varphi: D^b(\Ocat)\hookrightarrow D^b(\HC^G(\U)_0)$.
This extension is a full embedding, it is $\SBim^\wedge$-equivariant and intertwines
the duality functors.
\end{Lem}
\begin{proof}
The full embedding $\Ocat\operatorname{-tilt}\hookrightarrow \Hilt^G(\U)_0$
uniquely extends to a full embedding $K^b(\Ocat\operatorname{-tilt})\hookrightarrow
K^b(\Hilt^G(\U)_0)$. The source category is $D^b(\Ocat)$ because $\Ocat$ is an ideal
finite highest weight category. Since
there are no higher Ext's between objects in $\Hilt^G(\U)_0$, see Proposition
\ref{Prop:Ext_vanishing}, the natural functor $K^b(\Hilt^G(\U)_0)
\rightarrow D^b(\HC^G(\U)_0)$ is a full embedding. The remaining claims (the equivariance and
the compatibility with the dualities) follow from the construction.
 \end{proof}

 The following proposition is (iii) in the strategy above.

\begin{Prop}\label{Prop:exactness_property}
The full embedding $\varphi: D^b(\Ocat)\hookrightarrow D^b(\HC^G(\U)_0)$ is
t-exact (with respect to the perverse t-structure on the target)
and sends simple objects to simple objects.
\end{Prop}

Recall that $\Dual$ fixes all simples  in $\Perv(\HC^G(\U)_0)$, Proposition \ref{Prop:perverse_K0}
(and in $\Ocat$, Lemma \ref{Lem:dual_O}). Since the
embedding $D^b(\Ocat)\hookrightarrow D^b(\HC^G(\U)_0)$ intertwines the duality
functors, Proposition \ref{Prop:exactness_property} follows from the next lemma.

\begin{Lem}\label{Lem:self_dual_embedding}
Let $\mathcal{C},\mathcal{D}$ be triangulated categories equipped with homologically
finite t-structures with hearts $\mathcal{C}^\heartsuit, \mathcal{D}^\heartsuit$
and t-exact duality functors $\Dual_{\mathcal{C}},\Dual_{\mathcal{D}}$.
Let $\varphi:\mathcal{C}\rightarrow \mathcal{D}$ be an exact full embedding
such that $\varphi\circ \Dual_{\mathcal{C}}\cong \Dual_{\mathcal{D}}\circ \varphi$.
Assume, further, that all objects in $\mathcal{C}^\heartsuit, \mathcal{D}^\heartsuit$
have finite length and all simples in $\mathcal{D}^\heartsuit$ are fixed by $\Dual_{\mathcal{D}}$.
Then $\varphi$ is t-exact and sends simples to simples.
\end{Lem}
We will see in the proof that the condition that all objects in $\mathcal{C}^\heartsuit$ have finite length follows from the other conditions of the lemma.
\begin{proof}[Proof of Lemma \ref{Lem:self_dual_embedding}]
The proof is in several steps.

{\it Step 1}. First of all, note that for every $M\in \mathcal{D}^\heartsuit$ and every
simple $L$ with $M\twoheadrightarrow L$ we have a composed morphism $\kappa_L:M\twoheadrightarrow
L\xrightarrow{\sim} \Dual_{\mathcal{D}}L\hookrightarrow \Dual_{\mathcal{D}}M$. It is nonzero
because it has nonzero image.

{\it Step 2}. Assume $\varphi$ is not t-exact. This means there is $M'\in \mathcal{C}^\heartsuit$
such that $H^i(\varphi M')\neq 0$ for some $i\neq 0$. Replacing $M'$ with $\Dual_{\mathcal{C}}M'$ if needed
we can assume that $i>0$. Also we can assume that $H^j(\varphi M')=0$ for $j>i$ because the t-structure on
$\mathcal{D}$ is homologically finite. Pick a simple
object $L$ with $H^i(\varphi M')\twoheadrightarrow L$. Then $L\cong \Dual_{\mathcal{D}} L\hookrightarrow
\Dual_{\mathcal{D}}(H^i(\varphi M'))\cong H^{-i}(\varphi(\Dual_{\mathcal{C}}M'))$. So we get a nonzero
morphism $H^i(\varphi M')\rightarrow H^{-i}(\varphi(\Dual_{\mathcal{C}}M'))$. It follows that
there is a nonzero morphism $\varphi M'[i]\rightarrow \varphi(\Dual_{\mathcal{C}}M')[-i]$. Since $\varphi$ is a full embedding, this means that we have a nonzero morphism $M'[i]\rightarrow \Dual_{\mathcal{C}}M'[-i]$. Since both $M', \Dual_{\mathcal{C}}M'$ are in the heart of a t-structure and $i>0$, we arrive at a contradiction. This proves that $\varphi$ is t-exact.

{\it Step 3}. Now we show that $\varphi$ maps simple objects to simple objects.
Let $L'\in \mathcal{C}^\heartsuit$ be simple. Assume $\varphi L'$ is not simple.
Let $L$ be a simple in $\mathcal{D}^\heartsuit$ such that $\varphi L'\twoheadrightarrow
L$. Consider the nonzero morphism
$
\kappa_L:\varphi L'\rightarrow \Dual_{\mathcal{D}}\varphi L'\cong
\varphi \Dual_{\mathcal{C}}L'
$
 associated to $L$ as in Step 1. Since $\varphi$ is a full
embedding, $\kappa_L$ comes from a nonzero morphism
$\kappa'_L:L'\rightarrow \Dual_{\mathcal{C}}L'.$
Since $\Dual_{\mathcal{C}}L'$ is also simple, we conclude that $\kappa'_L$ is
an isomorphism. Hence $\kappa_L$ is an isomorphism. It follows that $\varphi L'\xrightarrow{\sim}
L$. We arrive at a contradiction with the assumption that $\varphi L'$ is not simple which finishes the proof of the lemma.
\end{proof}

Finally we proceed to step (iv) of the construction above.

\begin{Thm}\label{Thm:final_equivalence}
The full embedding $D^b(\Ocat)\hookrightarrow D^b(\HC^G(\U)_0)$ is a t-exact category equivalence.
\end{Thm}
\begin{proof}
Since we already know that  $\varphi$ is a full embedding, what we need to show that every simple in $\Perv(\HC^G(\U)_0)$ comes from a simple in $\Ocat$. Recall that $\varphi:\Ocat\rightarrow
\Perv(\HC^G(\U)_0)$ sends $T(1)=L(1)$ to $\U_0$ and is $\SBim^\wedge$-equivariant.
Also recall that by Proposition \ref{Prop:perverse_K0}, the class of $\U_0$ generates
$K_0(\Perv(\HC^G(\U)_0))$ under the action of $K_0(\SBim^{\wedge})$. It follows
that $\varphi$ induces an epimorphism $K_0(\Ocat)\twoheadrightarrow K_0(\Perv(\HC^G(\U)_0))$.
Since $\varphi$ maps simples to simples, we deduce that all simples
$\Perv(\HC^G(\U)_0)$ are in the image of $\varphi$. This finishes the proof.
\end{proof}

\subsection{Other equivalences}
Here we deduce some consequences from Theorem \ref{Thm:final_equivalence}.

Recall, Proposition
\ref{Prop:BR}, that we have a monoidal full embedding $\SBim^\wedge\hookrightarrow \Hilt^G(\U^{\wedge_0})$.
Denote it by $\varphi_\Ring$. The following result establishes Theorem \ref{Thm:additive}
from the introduction.

\begin{Thm}\label{Thm:HC-tilting_equiv}
The full embedding $\varphi_{\Ring}:\SBim^\wedge\hookrightarrow \Hilt^G(\U^{\wedge_0})$ is
an equivalence.
\end{Thm}
\begin{proof}
 Recall that $\varphi:K^b(\Ocat\operatorname{-tilt})\cong D^b(\Ocat)\hookrightarrow D^b(\HC^G(\U)_0)$
  from Proposition \ref{Prop:exactness_property} is an equivalence (Theorem \ref{Thm:final_equivalence}) that factors through
the full subcategory $K^b(\Hilt^G(\U)_0)\subset D^b(\HC^G(\U)_0)$.
Hence $K^b(\Ocat\operatorname{-tilt})\xrightarrow{\sim} K^b(\Hilt^G(\U)_0)$.
Recall that the latter functor is induced by the full embedding $\Ocat\operatorname{-tilt}
\hookrightarrow \Hilt^G(\U)_0$. It follows that this full embedding is an equivalence.

Now let $T_{\Ring}$ be an indecomposable object in $\Hilt^G(\U^{\wedge_0})$.
Then $T:=T_{\Ring}\otimes_{\Ring}\F$ is also indecomposable, this follows
from (1) of Corollary \ref{Cor:flat_hom}. So we have
an automatically indecomposable object $T'_{\Ring}\in \Ocat_{\Ring}\operatorname{-tilt}$
such that  $\varphi(T'_{\Ring}\otimes_{\Ring}\F)\cong T$. Then
$\varphi_{\Ring}(T'_{\Ring})\cong T_{\Ring}$. So the full embedding $\varphi_\Ring: \Ocat_{\Ring}\operatorname{-tilt}\rightarrow \Hilt^G(\U^{\wedge_0})$ is essentially
surjective and hence an equivalence.
\end{proof}

We also have the following equivalence between derived categories, Theorem \ref{Thm:derived}.

\begin{Thm}\label{Thm:derived_equivalence}
The equivalence $\varphi_{\Ring}$ from Theorem \ref{Thm:HC-tilting_equiv}
extends to $$D^b(\Ocat_{\Ring})\cong K^b(\SBim^\wedge)
\xrightarrow{\sim} D^b(\HC^G(\U^{\wedge_0})).$$
\end{Thm}
\begin{proof}
Thanks to Proposition \ref{Prop:BR} and (3) of Lemma \ref{Lem:HC-tilting_basic},
we have a full embedding $\SBim^{\wedge}\hookrightarrow \Hilt^G(\U^{\wedge_0})$.
It extends to $D^b(\Ocat_\Ring)\xrightarrow{\sim} K^b(\Hilt^G(\U^{\wedge_0}))$.
And thanks to (3) of Corollary \ref{Cor:flat_hom}, we have a full embedding
$K^b(\Hilt^G(\U^{\wedge_0}))\hookrightarrow D^b(\HC^G(\U^{\wedge_0}))$.
Denote the composite full embedding $D^b(\Ocat_\Ring)\hookrightarrow D^b(\HC^G(\U^{\wedge_0}))$
also by $\varphi_\Ring$.

So we need to show that every object $M_{\Ring}\in \HC^G(\U^{\wedge_0})$ lies in the image,
equivalently, thanks to Theorem \ref{Thm:HC-tilting_equiv}, in the full subcategory
$K^b(\Hilt^G(\U^{\wedge_0}))\subset D^b(\HC^G(\U^{\wedge_0}))$.
Let $f_1,\ldots,f_k$ denote parameters for $\Ring$, i.e., $\Ring=\F[[f_1,\ldots,f_k]]$.
We will use the ascending induction on $i$ to show that any $M\in \HC^G(\U^{\wedge_0})$
annihilated by $f_1,\ldots,f_{k-i}$ lies in
$K^b(\Hilt^G(\U^{\wedge_0}))$. We note that since $\varphi_{\Ring}$
is a full embedding, the image is closed under taking cones.

The base is $i=0$. Here $M\in \HC^G(\U)_0$. Then, thanks to Theorem
\ref{Thm:final_equivalence}, there is an object
$M'\in K^b(\Ocat\operatorname{-tilt})$ with $\varphi(M')\cong M$.
We can lift $M'$ to an object of $K^b(\Hilt^G(\U^{\wedge_0}))$
by replacing every term in the complex $M'$ by its Koszul resolution.
Denote the resulting object in $K^b(\Hilt^G(\U^{\wedge_0}))$ by $\tilde{M}'$.
The images of $\tilde{M}'$ and $M'$ in $D^b(\Ocat_{\Ring})$ are isomorphic.
It follows that $\varphi_{\Ring}(\tilde{M}')\cong M$.

Now suppose that we know that any HC bimodule
annihilated by $f_1,\ldots,f_{k-i}$ lies in the image of $K^b(\Hilt^G(\U^{\wedge_0}))$.
Since the image of  $K^b(\Hilt^G(\U^{\wedge_0}))$ is closed under taking
cones, all objects in $\HC^G(\U^{\wedge_0})$, where $f_1,\ldots,f_{k-i}$ act
nilpotently, lie in the image. Now let $M$ be an object in $\HC^G(\U^{\wedge_0})$
annihilated by $f_1,\ldots,f_{k-i-1}$. We want to show that it lies in the image of
$K^b(\Hilt^G(\U^{\wedge_0}))$. Set $f:=f_{k-i}$. The $f$-torsion part of $M$
lies in the image of $K^b(\Hilt^G(\U^{\wedge_0}))$. So we can assume that $M$
is torsion free over $\F[f]$.

For a finite poset ideal $I\subset W^{ea}$, we write $\Ocat_{\Ring}(I)$ for the full subcategory
of $\Ocat_{\Ring}$ generated by $T_{\Ring}(y)$ for $y\in I$.
This is a highest weight subcategory. Since $M/Mf$ lies in the image of $K^b(\Ocat_{\Ring}\operatorname{-tilt})$ by our inductive assumption,  we can find $I$ such that
$M/Mf$ lies in the image of $K^b(\Ocat_{\Ring}(I)\operatorname{-tilt})$.
We are going to show that $M$ actually lies in  the image of
$K^b(\Ocat_{\Ring}(I)\operatorname{-tilt})$.

Let $T$ denote the direct sum of all indecomposable tiltings in the image of
$\Ocat_{\Ring}(I)\operatorname{-tilt}$ in $\Hilt^G(\U)_0$. Set
$$\Asf_T:=\End_{\Ocat_{\Ring}}(T)^{opp}.$$ This $\Ring$-algebra is a free finite rank
$\Ring$-module. It has finite homological dimension because
$\Asf_T\operatorname{-mod}$ is a highest weight category (with poset $I^{opp}$).
We have the $\Ring$-linear functor $\mathcal{F}:=R\Hom(T,\bullet):D^b(\HC^G(\U^{\wedge_0}))
\rightarrow D^b(\Asf_T\operatorname{-mod})$. Since $\Asf$ has finite homological
dimension, the functor $\mathcal{F}$ has the left adjoint and  right inverse $\mathcal{G}:=
T\otimes^L_{\Asf_T}\bullet$. Apply the adjunction counit to $M$ and complete the resulting
morphism to an exact triangle
\begin{equation}\label{eq:exact_triangle}
\mathcal{G}\mathcal{F}M\rightarrow M\rightarrow N\xrightarrow{+1}.
\end{equation}
Note that by the construction, we have $\mathcal{F}(N)=0$. Now apply $\bullet\otimes^L_{\F[f]}\F$
to (\ref{eq:exact_triangle}). Since $\mathcal{F},\mathcal{G}$ are $\Ring$-linear and hence
$\F[f]$-linear and $M$ is flat over $\F[f]$, we get
$$\mathcal{G}\mathcal{F}(M/Mf)\rightarrow M/Mf\rightarrow N\otimes^L_{\F[f]}\F\xrightarrow{+1}.$$
By our assumption, $M/Mf$ lies in the essential image $K^b(\Ocat_{\Ring}(I)\operatorname{-tilt})$,
which coincides with the essential image of $\mathcal{G}$. It follows that
$\mathcal{G}\mathcal{F}(M/Mf)\xrightarrow{\sim} M/Mf$. Hence $N\otimes^L_{\F[f]}\F=0$.
We claim that the last equality implies $N=0$. Indeed, let $j$ be maximal such that $H^j(N)\neq \{0\}$.
Then $H^j(N\otimes^L_{\F[f]}\F)=H^j(N)/ H^j(N)f$. Note that $H^j(N)$ is a finitely generated $\U^{\wedge_0}$-module. Its support in the spectrum of the center is closed,
hence $H^j(N)/ H^j(N)f\neq \{0\}$. This contradicts
$N\otimes^L_{\F[f]}\F=0$ and finishes the induction step and therefore the proof.
\end{proof}

%

\section{Modular category $\Ocat$}
\subsection{Definition and equivalent characterization}\label{SS_O_cl_basic}
Fix a Borel subgroup $B\subset G$.
We consider the category
$\Ocat^{cl}$ (``cl'' from ``classical'' as opposed to the Soergel
type categories from Section \ref{S_Soergel_O})
of all finitely generated $\g$-modules with a strongly $B$-equivariant
structure.

For example, for every $\lambda\in \Lambda$, we have the Verma module
$\Delta^{cl}(\lambda)\in \Ocat^{cl}$. As usual, it is defined as
$U(\g)\otimes_{U(\mathfrak{b})}\F_\lambda$, where $B$ acts on the 1-dimensional space
$\F_\lambda$ by the character $\lambda$, while on $\Delta^{cl}(\lambda)$
we have the tensor product action.

We remark that $\Delta^{cl}(\lambda)$ has infinite length,
and the dual Verma module $\nabla^{cl}(\lambda)$ (constructed in the usual fashion) is not in
$\Ocat^{cl}$ but rather in its ind-completion: it is the inductive limit of the duals of
finite dimensional quotients of $\Delta^{cl}(\lambda)$.

As usual, $\Ocat^{cl}$ decomposes as the sum of infinitesimal blocks $\bigoplus_{[\lambda]}\Ocat^{[\lambda]}$, where the sum is taken over
the $W^{ea}$-orbits in $\Lambda$ (with the action of the lattice part
rescaled $p$ times). The objects in $\Ocat^{[\lambda]}$
are exactly the modules in $\Ocat^{cl}$ with generalized central character
$\lambda \operatorname{mod} p$.

Here is an equivalent construction. Recall that $\mathcal{B}$ stands for the flag variety $G/B$.
So we can consider the category $\Coh^{G}(D_{\mathcal{B}})$ of weakly equivariant
$D_{\B}$-modules (that are quasi-coherent over $\mathcal{O}_{\mathcal{B}}$ and are locally
finitely generated over $D_{\mathcal{B}}$). The following result is classical.

\begin{Lem}\label{eq:O_D_mod}
Taking the fiber at $1B\in \mathcal{B}$ defines an equivalence
$$\Coh^{G}(D_{\mathcal{B}})\xrightarrow{\sim}\Ocat^{cl}.$$
\end{Lem}

For example, the regular module $D_{\mathcal{B}}$ corresponds to the Verma module $\Delta^{cl}(0)$.

We will need a slight modification of this equivalence. Let $K_{\mathcal{\B}}$ denote the canonical
bundle on $\mathcal{B}$ with its natural $G$-equivariant structure. The functor
$K_{\mathcal{B}}\otimes_{\mathcal{O}_{\mathcal{\B}}}\bullet$ defines an equivalence between $\Coh^{G}(D_{\mathcal{B}})$
and $\Coh^{G}(D_{\mathcal{B}}^{opp})$. So we get an equivalence
$$\Coh^{G}(D_{\mathcal{B}}^{opp})\xrightarrow{\sim}\Ocat^{cl}.$$
Under this equivalence the regular right module $D_{\mathcal{B}}$ goes to
the Verma module $\Delta^{cl}(-2\rho)$.

\subsection{Main equivalence}
Our main result concerning $\Ocat^{[0]}$ is as follows.
Let $\pi: \tilde{\Nilp}^{(1)}\rightarrow \g^{(1)*}$ be the Springer map.
Recall, Section \ref{SS:tilt_NC}, that we have the $\F[\g^{*(1)}]$-algebra $\Acal_\Ring$ acted on by
$G^{(1)}$. So we can consider its
pullback $\pi^*\Acal_\Ring$ and the category $\Coh^{G^{(1)}}(\pi^*\Acal_\Ring)$.

\begin{Thm}\label{Thm:Ocat_cl_main}
We have an equivalence $\Ocat^{[0]}\xrightarrow{\sim} \Coh^{G^{(1)}}(\pi^*\Acal_\Ring)$
of abelian categories.
\end{Thm}

We now explain how this equivalence is constructed. First, we establish an
equivalence
\begin{equation}\label{eq:cat_O_equiv1}
D^b(\Ocat^{[0]})\xrightarrow{\sim} D^b(\HC^G(\U)_0).
\end{equation}
It is also known that -- we'll recall why -- that
\begin{equation}\label{eq:cat_O_equiv3}
D^b(\Acal_\Ring\otimes_{\F[\g^{*(1)}]}\Acal^{opp}\operatorname{-mod}^{G^{(1)}})
\xrightarrow{\sim} D^b(\Coh^{G^{(1)}}(\pi^*\Acal_\Ring)).
\end{equation}

Composing (right to left) (\ref{eq:cat_O_equiv1}),
(\ref{eq:bimod_equiv_abelian_specialized}), and (\ref{eq:cat_O_equiv3}), we get
\begin{equation}\label{eq:composed_equiv_O}
D^b(\Ocat^{[0]})\xrightarrow{\sim} D^b(\Coh^{G^{(1)}}(\pi^*\Acal_\Ring)).
\end{equation}
Our last step will be to show that this equivalence is t-exact.

The following proposition (which should be thought of as a characteristic $p$
version of the classical Bernstein-Gelfand equivalence in characteristic $0$)
constructs the quasi-inverse of (\ref{eq:cat_O_equiv1}). It is a more precise version of
Proposition \ref{Prop:O_HC}.

\begin{Prop}\label{Prop:BG_derived}
The functor $\mathcal{M}\mapsto \mathcal{M}\otimes^L_{\U_{0}}\Delta^{cl}(-2\rho)$
is a category equivalence $$D^b(\HC^G(\U)_0)\xrightarrow{\sim} D^b(\Ocat^{cl})$$
that restricts to $D^b(\HC(\U)_0)\xrightarrow{\sim} D^b(\Ocat^{[0]})$.
\end{Prop}
\begin{proof}
Thanks to \cite[Theorem 3.2]{BMR} (more, precisely, its straightforward analog for
right modules), we have mutually quasi-inverse equivalences
$$R\Gamma: D^b(\Coh(D_{\mathcal{B}}^{opp}))\rightleftarrows
D^b(\U_{0}^{opp}\operatorname{-mod}): \bullet\otimes^L_{\U_{0}}D_{\mathcal{B}}.$$
The functors lift to an adjoint pair of functors between
the categories of weakly equivariant modules and hence give
mutually quasi-inverse equivalences
$$D^b(\Coh^{G}(D_{\mathcal{B}}^{opp}))\rightleftarrows
D^b(\U_{0}^{opp}\operatorname{-mod}^{G}).$$
Then we have the (t-exact) equivalence
$$D^b(\Coh^{G}(D_{\mathcal{B}}^{opp}))\xrightarrow{\sim} D^b(\Ocat^{cl})$$
of restricting to the point $1B$, see Section \ref{SS_O_cl_basic}.
Since it sends $D_{\mathcal{B}}$ to $\Delta^{cl}(-2\rho)$,
the composition
$$D^b(\U_{0}^{opp}\operatorname{-mod}^{G})\xrightarrow{\sim}
D^b(\Coh^{G}(D_{\mathcal{B}}^{opp}))\xrightarrow{\sim} D^b(\Ocat^{cl})$$
is $\bullet\otimes^L_{\U_{0}}\Delta^{cl}(-2\rho)$. This establishes the first
claim.

To establish that $D^b(\HC^G(\U)_0)\xrightarrow{\sim} D^b(\Ocat^{[0]})$
we notice that the functor $\bullet\otimes^L_{\U_{0}}\Delta^{cl}(-2\rho)$ sends the source to the target
because it preserves generalized central characters.
\end{proof}

In what follows we always identify $D^b(\Ocat^{[0]})$ and $D^b(\HC^G(\U)_0)$ using the equivalences of the lemma.

Now we recall equivalence (\ref{eq:cat_O_equiv3}). We have the equivalence
\begin{equation}\label{eq:derived_equivalence_coh_semicoh}
R\pi_{1*}([\Tcal_{\h}\boxtimes \mathcal{O}_{\tilde{\g}^{(1)}}]|_{\St}\otimes\bullet):
D^b(\Coh^{G^{(1)}}\St^{(1)})\xrightarrow{\sim} D^b(\Coh^{G^{(1)}}(\pi^*\Acal_\Ring),
\end{equation}
where we write $\pi_1$ for the projection $\St^{(1)}\rightarrow \tilde{\g}^{(1)}$,
see \cite[Section 2.2]{BLin}. Thanks to the equivalence (\ref{eq:derived_equiv_bimod_specialized}),
we get an equivalence
\begin{equation}\label{eq:der_equiv_A_new}
R\Gamma(\mathcal{T}^*\otimes\bullet): D^b(\Coh^{G^{(1)}}(\pi^*\Acal_\Ring))\xrightarrow{\sim}
D^b(\Acal_\Ring\otimes_{\F[\g^{*(1)}]}\Acal^{opp}\operatorname{-mod}^{G^{(1)}}).
\end{equation}
Here we view the source category as that of $G^{(1)}$-equivariant $\Acal_\Ring\otimes_{\F[\g^{(1)*}]}
\mathcal{O}_{\tilde{\Nilp}^{(1)}}$-modules, and tensoring with $\mathcal{T}^*$ is in the 2nd factor.

So we get the composed equivalence (\ref{eq:composed_equiv_O}).

\begin{proof}[Proof of Theorem \ref{Thm:Ocat_cl_main}]
To prove the theorem it remains to show that (\ref{eq:composed_equiv_O})
is t-exact with respect to the tautological t-structures. Equivalently,
we need to show that the derived equivalence induced by (\ref{eq:bimod_equiv_abelian_specialized})
intertwines
\begin{itemize}
\item
The image of $D^{b,\leqslant 0}(\Ocat^{[0]})$ (the negative part of the t-structure) in $D^b(\HC^G(\U)_0)$,
\item and the image of $D^{b,\leqslant 0}(\Coh^{G^{(1)}}(\pi^*\Acal_\Ring))$
in  $D^b(\Acal_\Ring\otimes_{\F[\g^{*(1)}]}\Acal^{opp}\operatorname{-mod}^{G^{(1)}})$.
\end{itemize}
This is done in several steps.
Below in the proof we identify all derived categories involved with
$D^b(\Coh^{G^{(1)}}\St^{(1)})$.

{\it Step 1}. Fix a strictly dominant weight $\lambda$ (so that the sheaf $\mathcal{O}(\lambda)$
on $\mathcal{B}$ is ample). We claim that the images of both
$D^{b,\leqslant 0}(\Ocat^{[0]})$ and $D^{b,\leqslant 0}(\Coh^{G^{(1)}}(\pi^*\Acal_\Ring))$
in $D^b(\Coh^{G^{(1)}}\St^{(1)})$ coincide with the full subcategory of all
objects $\mathcal{F}\in D^b(\Coh^{G^{(1)}}\St^{(1)})$ satisfying the following condition:
\begin{itemize}
\item[($\heartsuit$)] There is $n_0\in \mathbb{Z}_{>0}$ such that for all $n>n_0$ we have
$$\mathcal{F}\otimes \pi_2^* \mathcal{O}(n\lambda)\in
D^{b,\leqslant 0}(\Acal_\Ring\otimes_{\F[\g^{*(1)}]}\Acal^{opp}\operatorname{-mod}^{G^{(1)}})=
D^{b,\leqslant 0}(\HC^G(\U)_0).$$
\end{itemize}
 We will see below that this claim is essentially the Serre vanishing
theorem.

{\it Step 2}.
We start by proving the claim  that the image of $D^{b,\leqslant 0}(\Coh^{G^{(1)}}(\pi^*\Acal_\Ring))$
consists of the objects satisfying ($\heartsuit$), where it is easier. 
We write
$\mathcal{H}^i$ for the $i$th cohomology sheaf (for the natural t-structure
on $D^{b}(\Coh^{G^{(1)}}(\pi^*\Acal_\Ring))$).
Take $\mathcal{F}\in D^b(\Coh^{G^{(1)}}(\pi^*\Acal_\Ring))$. We can find $n_0\in \mathbb{Z}_{>0}$
such that $\mathcal{H}^i(\mathcal{F})\otimes \mathcal{O}(n\lambda)$ has no higher cohomology
and is generated by global sections for all $n>n_0$ and all $i$. It follows that $\mathcal{F}\in D^{b,\leqslant 0}(\Coh^{G^{(1)}}(\pi^*\Acal_\Ring))$ if and only if
$$R\Gamma(\mathcal{F}\otimes \mathcal{T}^*\otimes \mathcal{O}(n\lambda))\in D^{b,\leqslant 0}(\Acal_\Ring\otimes_{\F[\g^{*(1)}]}\Acal^{opp}\operatorname{-mod}^{G^{(1)}}), \forall n>n_0.$$
Now the claim that the image of $D^{b,\leqslant 0}(\Coh^{G^{(1)}}(\pi^*\Acal_\Ring))$
consists of all objects satisfying ($\heartsuit$)
follows from two observations. First, the composition of the derived
equivalences (\ref{eq:derived_equivalence_coh_semicoh}) and (\ref{eq:der_equiv_A_new})
is the derived equivalence (\ref{eq:derived_equiv_bimod_specialized}). Second,
(\ref{eq:derived_equivalence_coh_semicoh}) intertwines twisting by
$\pi_2^*(\mathcal{O}(n\lambda))$ in the source with twisting by $\mathcal{O}(n\lambda)$
in the target.

{\it Step 3}.
Now we proceed to proving the claim about the image of $D^{b,\leqslant 0}(\Ocat^{[0]})$: that it
consists of all objects satisfying ($\heartsuit$). Since we have no direct construction
of an equivalence between  $D^{b,\leqslant 0}(\Ocat^{[0]})$ and
$D^b(\Coh^{G^{(1)}}\St^{(1)})$, a proof is somewhat more involved. First, note that
an argument similar to the previous paragraph implies that the image of
$$D^{b,\leqslant 0}(\Ocat^{cl})= D^{b,\leqslant 0}(\operatorname{Coh}^G(D^{opp}_{G/B}))$$
in $D^b(\U_0^{opp}\operatorname{-mod}^G)$ is characterized by the direct analog of ($\heartsuit$)
where we twist with $\mathcal{O}_{\mathcal{B}}(np\lambda)$ for some $n>n_1$.
Now we describe this twist in terms of HC bimodules. Observe, that if $n$ is sufficiently large, then $R\Gamma(\mathcal{O}_{\mathcal{B}}(np\lambda)\otimes_{\mathcal{O}_{\mathcal{B}}}D_{\mathcal{B}})$
is concentrated in homological degree $0$ (again, this is Serre vanishing theorem applied on the variety
$\tilde{\Nilp}^{(1)}$). Also $\Gamma(\mathcal{O}_{\mathcal{B}}(np\lambda)\otimes_{\mathcal{O}_{\mathcal{B}}}D_{\mathcal{B}})$ is an object in $\HC^G(\U_0)$ that will be denoted
by $\mathcal{M}_n$. We remark that
\begin{equation}\label{eq:rgamma_twist}
R\Gamma(\mathcal{F}\otimes_{\mathcal{D}_{\mathcal{B}}}(\mathcal{O}_{\mathcal{B}}(np\lambda)\otimes_{\mathcal{O}_{\mathcal{B}}}D_{\mathcal{B}}))
\cong R\Gamma(\mathcal{F})\otimes^L_{\U_0}\mathcal{M}_n.
\end{equation}
Indeed, set $\mathcal{G}:=R\Gamma(\mathcal{F})$ and write $\operatorname{Loc}$ for the quasi-inverse of $R\Gamma$, the functor $\bullet^L_{\U_0}\mathcal{D}_{\mathcal{B}}$. The isomorphism (\ref{eq:rgamma_twist}) is equivalent to 
$$\operatorname{Loc}(\mathcal{G})\otimes_{\mathcal{D}_{\mathcal{B}}}(\mathcal{O}_{\mathcal{B}}(np\lambda)
\otimes_{\mathcal{O}_{\mathcal{B}}}D_{\mathcal{B}}))\xrightarrow{\sim}
\operatorname{Loc}(\mathcal{G}\otimes_{\U_0}\mathcal{M}_n).$$
The restrictions of the two functors to the category of free $\U_0$-modules are isomorphic.
So they are isomorphic on $D^-(\U_0^{opp}\operatorname{-mod})$ and hence
on $D^b(\U_0^{opp}\operatorname{-mod}^G)$. 

Thanks to (\ref{eq:rgamma_twist}),  the image of $D^{b,\leqslant 0}(\operatorname{Coh}^G(D^{opp}_{G/B}))$
in $D^b(\U_0^{opp}\operatorname{-mod}^G)$ coincides with the full subcategory of all objects $\mathcal{F}$
such that $\mathcal{F}\otimes^L_{\U_0}\mathcal{M}_n\in D^{b,\leqslant 0}(\U_0^{opp}\operatorname{-mod}^G)$. Note that all objects in
$\U_0^{opp}\operatorname{-mod}^G$ admit finite resolutions by HC modules that are projective
$\U_0^{opp}$-modules (because $\U_0^{opp}$ has finite homological dimension, and every object
in $\U_0^{opp}\operatorname{-mod}^G$ admits an epimorphism from a module of the form $V\otimes \U_0$, where $V$ is a finite dimensional $G$-module). It follows that the endo-functor
$\bullet\otimes^L_{\U_0}\mathcal{M}_n$ of $D^b(\U_0^{opp}\operatorname{-mod}^G)$
is the left derived functor of $\bullet\otimes_{\U_0}\mathcal{M}_n$. Note that
$\bullet\otimes_{\U_0}\mathcal{M}_n$ clearly preserves $\HC^G(\U)_0$.

{\it Step 4}.
We need to determine the endofunctor of $\Acal_\Ring\otimes_{\F[\g^{*(1)}]}\Acal^{opp}\operatorname{-mod}^{G^{(1)}}$ corresponding to the  functor $\bullet\otimes_{\U_0}\mathcal{M}_n$ via the equivalence (\ref{eq:bimod_equiv_abelian_specialized}). Since (\ref{eq:bimod_equiv_abelian_specialized}) is a right module equivalence for the monoidal equivalence
(\ref{eq:bimod_equiv_abelian_specialized2}).

Consider the equivalence
\begin{equation}\label{eq:one_more_equiv1}
D^b(\U_0\otimes^L_{\F[\g^{(1)*}]}\U_0^{opp}\operatorname{-mod}^G)\xrightarrow{\sim}
D^b(\operatorname{Coh}^{G}(\tilde{\Nilp}^{(1)}\times^L_{\g^{*(1)}}\tilde{\Nilp}^{(1)}))\xrightarrow{\sim}
D^b(\Acal\otimes^L_{\F[\g^{(1)*}]}\Acal^{opp}\operatorname{-mod}^G),
\end{equation}
analogous to (\ref{eq:bimod_equiv_main}). Then (\ref{eq:bimod_equiv_abelian_specialized2}) is the restriction of the composed equivalence in (\ref{eq:one_more_equiv1}) to the hearts.
The bimodule $\mathcal{M}_n$ is the image
of the line bundle $\mathcal{O}(n\lambda)$ (with its natural $G$-equivariant structure)
on the diagonal copy of $\tilde{\Nilp}^{(1)}$ under the quasi-inverse of the first equivalence
in (\ref{eq:one_more_equiv1}). So the image of $\mathcal{M}_n$ under the composed equivalence in
(\ref{eq:one_more_equiv1}) is $\Gamma(\mathcal{E}nd(\mathcal{T})\otimes \mathcal{O}(n\lambda))$.
Our conclusion is that  the endofunctor of $\Acal_\Ring\otimes_{\F[\g^{*(1)}]}\Acal^{opp}\operatorname{-mod}^{G}$  corresponding to
the endofunctor $\bullet\otimes_{\U_0}\mathcal{M}_n$ of $\U\otimes_{\F[\g^{*(1)}]}\U_0^{opp}\operatorname{-mod}^G$
is $$\bullet\otimes_{\Acal}\Gamma(\mathcal{E}nd(\mathcal{T})\otimes \mathcal{O}(n\lambda)).$$

Under the derived equivalence
$$D^b(\Coh^{G^{(1)}}(\tilde{\Nilp}^{(1)}\times^L_{\g^{(1)*}}\tilde{\Nilp}^{(1)}))\xrightarrow{\sim}
D^b(\Acal\otimes^L_{\F[\g^{(1)*}]}\Acal^{opp}\operatorname{-mod}^{G^{(1)}})$$
the bimodule $\Gamma(\mathcal{E}nd(\mathcal{T})\otimes \mathcal{O}(n\lambda))$  corresponds still to $\mathcal{O}(n\lambda)$ (with its natural
$G^{(1)}$-equivariant structure) on the diagonal copy of $\tilde{\Nilp}^{(1)}$.
The convolution with this sheaf on the right, an endofunctor of
$D^b(\Coh^{G^{(1)}}\St^{(1)})$ is the twist by $\mathcal{O}(n\lambda)$ in the second
component.

So, under the equivalence $D^b(\HC^G(\U_0))\xrightarrow{\sim}D^b(\Coh^{G^{(1)}}\St^{(1)})$
the functor $\otimes^L_{\U_0}\mathcal{M}_n$ becomes the twist with $\mathcal{O}(n\lambda)$ in the second
component.
It follows that under the equivalence
(\ref{eq:HC_derived_specialized}), the image of $D^{b,\leqslant 0}(\Ocat^{[0]})$
in $D^b(\Coh^{G^{(1)}}\St^{(1)})$ consists of all objects satisfying ($\heartsuit$).
This finishes the proof of the theorem.
\end{proof}

\begin{Rem}\label{Rem:lattice_equivariance}
Note that the lattice $\Lambda$ acts on $\Ocat^{[0]}$ (by twisting with characters of $B^{(1)}$)
and on $\Coh^{G^{(1)}}(\pi^*\Acal_\Ring)$ by twisting with the $G^{(1)}$-equivariant line bundles.
The equivalence in Theorem \ref{Thm:Ocat_cl_main} intertwines these actions. Namely, Step 4
of the proof shows that the twists with sufficiently ample line bundles/ sufficiently dominant
characters are intertwined. The claim in general follows.
\end{Rem}

\subsection{Localizations of Verma modules}
Consider the derived equivalence
\begin{equation}\label{eq:O_local}
\mathcal{F}:D^b(\Ocat^{[0]})\xrightarrow{\sim} D^b(\Coh^{G^{(1)}}\St^{(1)}),
\end{equation}
the composition of the quasi-inverse of (\ref{eq:derived_equivalence_coh_semicoh})
and (\ref{eq:composed_equiv_O}). Our task in this  section is to compute the
images of the Verma modules $\Delta^{cl}(\lambda)\in\Ocat^{[0]}$ under $\mathcal{F}$. Note that the set of
highest weights of these Vermas is identified with $W^{ea}$ via $x\mapsto x^{-1}\cdot (-2\rho)$,
so that for $x=wt_\lambda$ we have $x^{-1}\cdot (-2\rho)=-w^{-1}\rho-\rho-p\lambda$.

To state the answer, we need to recall the braid group action from \cite{BR_braid}.
The first step is to construct a homomorphism from the extended affine braid group
$\mathsf{Br}^{ea}$ to the group of isomorphism classes of invertible objects
in $D^b(\Coh^G(\St_\h))$ (of course, the same holds after applying the Frobenius twist).
The homomorphism sends $\lambda\in \Lambda\subset \mathsf{Br}^{ea}$
to the sheaf $\mathcal{O}(-\lambda)$ on the diagonal copy of $\tilde{\g}$
(to be denoted by $\tilde{\g}_\Delta$).

The images of the generators $T_s$ of the finite braid group
inside $\mathsf{Br}^{ea}$ are determined as follows. Let
$\g^{*,rs}$ denote the locus of regular semisimple elements
and $\St_\h^{rs}$ be its preimage in $\St_\h$. For a simple reflection $s\in W$ consider the locus in
$\St_\h^{rs}$, where the two Borel subalgebras are in relative position $s$. Denote it by $Z_{s}^{rs}$.
Let $Z_{s}$ denote the Zariski closure of $Z_{s}^{rs}$ in
$\St_\h$. The element $T_s$ is sent to $\Str_{Z_{s}}$.

We will need an exact sequence obtained in \cite[Section 1.10]{BR_braid}.
For a simple reflection $s$ in $W$ we consider the scheme $\tilde{\g}_{s}$
constructed as follows. Let $\mathcal{P}_{s}$ denote the partial flag variety,
the quotient of $G$ by the minimal parabolic subgroup $P_{s}$ corresponding to $s$.
It is a closed subscheme in $\g^*\times \mathcal{P}_{s}$ consisting
of all pairs $(\alpha,\mathcal{F})$ such that the element of $\g$ corresponding
to $\alpha$ (under the identification coming from the Killing form) lies in the parabolic
subalgebra corresponding to $\mathcal{F}$. Note that we have a natural morphism
$\tilde{\g}\rightarrow \tilde{\g}_{s}$ and so can form the fiber product
$\tilde{\g}\times_{\tilde{\g}_{s}}\tilde{\g}$.
By \cite[(1.10.1)]{BR_braid}, we get a short exact sequence of coherent sheaves
\begin{equation}\label{eq:SES_braid_generator}
0\rightarrow \Str_{\tilde{\g}_{\Delta}}\rightarrow
\Str_{\tilde{\g}\times_{\tilde{\g}_s}\tilde{\g}}
\rightarrow \Str_{Z_{s}}\rightarrow 0.
\end{equation}

Note that $D^b(\Coh^{G}\St_\h)$
acts on $D^b(\Coh^{G}\St)$ by convolutions. This gives the first (left) action
of $\Br^{ea}$ on $D^b(\Coh^{G}\St)$. We remark that in this construction we can replace
$\St_\h$ with its completed version $\St_\Ring$.

We also have a commuting right action. Consider the derived scheme $\tilde{\Nilp}\times^L_{\g^*}\tilde{\Nilp}$.
It makes sense to speak about the derived category of $G$-equivariant
coherent sheaves on this scheme, to be denoted by
$D^b(\Coh^{G}\tilde{\Nilp}\times^L_{\g^*}\tilde{\Nilp})$.
It acts on $D^b(\Coh^{G}\St)$ by convolutions from the right.
Following \cite{BR_braid}, there is a homomorphism from $\Br^{ea}$ to the group of
invertible objects in $D^b(\Coh^{G}\tilde{\Nilp}\times^L_{\g^*}\tilde{\Nilp})$.
Namely, note that we have a natural morphism from the usual fiber product
$\tilde{\Nilp}\times_{\g^*}\tilde{\Nilp}$ to the derived
tensor product. So we can view objects of $\Coh^{G}(\tilde{\Nilp}\times_{\g^*}\tilde{\Nilp})$
as objects of
$D^b(\Coh^{G}\tilde{\Nilp}\times^L_{\g^*}\tilde{\Nilp})$.

An element $\lambda\in \Lambda\subset \Br^{ea}$ is sent to $\mathcal{O}_{\tilde{\Nilp}_{\Delta}}(-\lambda)$, while $T_s$ is sent to the structure sheaf of the scheme-theoretic
intersection $Z_{s}\cap \St$ (that is actually a subscheme in
$\tilde{\Nilp}\times_{\g^*}\tilde{\Nilp}$).

\begin{Prop}\label{Prop:Verma_localization}
For $x=wt_\lambda$, we have
$$\mathcal{F}\left(\Delta^{cl}(x^{-1}\cdot (-2\rho))\right)\cong T_{w^{-1}}^{-1}\mathcal{O}_{\tilde{\Nilp}^{(1)}_{\Delta}}(-\lambda).$$
\end{Prop}
\begin{proof}
The proof of this proposition is in two steps.

{\it Step 1.} The affine braid group $\Br^{ea}$ acts on $D^b(\Ocat^{[0]})$ by the classical wall-crossing functors. We claim that $\mathcal{F}$ intertwines the action of the finite braid group $\Br_W$ on
$D^b(\Ocat^{[0]})$ and the action of $\Br_W$ on $D^b(\Coh^{G^{(1)}}\St^{(1)})$ from the left
(the claim should be true for the whole group $\Br^{ea}$ but we do not need this).

The action of $T_s$ on $D^b(\Ocat^{[0]})$ is by tensoring with the cone of $\mathsf{1}\rightarrow \Theta_s$, where $\mathsf{1}$ denotes the monoidal unit, and $\Theta_s$ is the classical reflection functor. The functor $\Theta_s$ is given by tensoring with the Harish-Chandra bimodule $B_s^{HC}$ from Section \ref{SS_BR_HC}.
The action of $T_s$ on
$D^b(\Coh^{G^{(1)}}\St^{(1)})$ is by convolving with the cone of
$$\mathsf{1}\rightarrow \operatorname{Spec}\Ring\times_{\h^{*(1)}}[\Str_{\tilde{\g}^{(1)}\times_{\tilde{\g}^{(1)}_s}\tilde{\g}^{(1)}}],$$
compare to (\ref{eq:SES_braid_generator}). Under the first equivalence in  (\ref{eq:derived_equiv_bimod}),
the target sheaf is sent to
\begin{equation}\label{eq:refl_bimod_NC_Springer}
R\Gamma\left(\operatorname{Spec}\Ring\times_{\h^{*(1)}}[\Str_{\tilde{\g}^{(1)}\times_{\tilde{\g}^{(1)}_s}\tilde{\g}^{(1)}}],
\mathcal{T}_\h\otimes \mathcal{T}_\h^*\right).
\end{equation}
Note that, by the construction in
\cite{BM},  (\ref{eq:refl_bimod_NC_Springer}) is in homological degree $0$ and  is flat as a right $\Acal^{(1)}_\Ring$-module (see (c) of Corollary in \cite[Section 2.3.1]{BM}), hence as a $\F[\g^{*(1)\wedge}]$-module. The bimodule
$B_s^{HC}$ is flat over $\F[\g^{*(1)\wedge}]$ as well.

Now we show that (\ref{eq:bimod_equiv_main}) sends $B_s^{HC}$ to (\ref{eq:refl_bimod_NC_Springer}).
Since both are flat over $\F[\g^{*(1)\wedge}]$, it is enough to establish the claim that the former
module is sent to the latter under our equivalence after restricting to
 $\g^{*(1),reg}$. Over that locus both categories in (\ref{eq:bimod_equiv_main}) are equivalent
to $\Coh^{G^{(1)}}(\St^{(1)}_\Ring)$, and these equivalences are intertwined
by  (\ref{eq:bimod_equiv_main}). Further, note that
$\Coh^{G^{(1)}}(\St^{(1),reg}_\Ring)$ is equivalent to
the category $\operatorname{Rep}(\mathfrak{J}^\wedge)$ considered in
Section \ref{SS_BR_AS}. The image of (\ref{eq:refl_bimod_NC_Springer})
in  $\Coh^{G^{(1)}}(\St^{(1),reg}_\Ring)$ is the structure sheaf on the intersection of
$\St^{(1),reg}_\Ring$ with $(\tilde{\g}\times_{\tilde{\g}_s}\tilde{\g})^{(1)}$.
Under the equivalence with $\operatorname{Rep}(\mathfrak{J}^\wedge)$ (under the restriction to $S^{(1)}$),
this structure sheaf goes to the image of $B^{AS}_s$ under the full
embedding of Lemma \ref{Lem:full_embedding_geometric}. The same is true
for the image of $B^{HC}_s$, see the proof of Proposition \ref{Prop:BR}.
It follows that (\ref{eq:bimod_equiv_main}) indeed sends $B_s^{HC}$ to (\ref{eq:refl_bimod_NC_Springer}).
Therefore, under the equivalence of Theorem \ref{Thm:derived_loc_HC}, $B^{HC}_s$
goes to the structure sheaf of $(\tilde{\g}\times_{\tilde{\g}_s}\tilde{\g})^{(1)}\cap \St^{(1)}_\Ring$.

It remains to show that, under the same equivalence, the cone of $\mathsf{1}\rightarrow B^{HC}_s$
goes to the cone of $$\mathsf{1}\rightarrow \Str_{(\tilde{\g}\times_{\tilde{\g}_s}\tilde{\g})^{(1)}\cap \St^{(1)}_\Ring}.$$
This will follow if we check $$\Hom_{\HC^G(\U^{\wedge_0})}(\mathsf{1},B^{HC}_s)\cong \Ring,$$
as the adjunction units generate the corresponding Hom $\Ring$-modules. Note that both
objects in the Hom above are HC tilting. So the Hom $\Ring$-module is the same as between
their images in the equivalent category $\SBim^\wedge$ (the equivalence has been established in
Theorem \ref{Thm:HC-tilting_equiv}). There, it follows from the definition.

This completes the claim of Step 1.

{\it Step 2}. The claim that $\Delta^{cl}(-2\rho)$ goes to $\Str_{\tilde{\Nilp}^{(1)},\Delta}$ follows from the construction of
the equivalence $D^b(\Ocat^{[0]})\xrightarrow{\sim} D^b(\Coh^{G^{(1)}}\St^{(1)})$. Indeed,
under the equivalence $D^b(\Ocat^{[0]})\xrightarrow{\sim} D^b(\HC^G(\U)_0)$, the module
$\Delta^{cl}(-2\rho)$ goes to $\U_0$, see Proposition \ref{Prop:BG_derived}.
The construction of Remark \ref{Rem:derived_loc_HC} shows that
under the equivalence $$D^b(\HC^G(\U)_0)\xrightarrow{\sim} D^b(\Acal_\Ring\otimes_{\F[\g^{*(1)}]}\Acal^{opp}\operatorname{-mod}^{G^{(1)}})$$
the bimodule $\U_0$ goes to the bimodule $\Acal$. And that bimodule goes to
$\Str_{\tilde{\Nilp}^{(1)},\Delta}\in D^b(\Coh^{G^{(1)}}\St^{(1)})$.

According to Remark \ref{Rem:lattice_equivariance},
the equivalence $\mathcal{F}$ intertwines the actions of $\Lambda$. So it is enough
to prove the claim of the proposition when $x=w\in W$. Step 1 combined with the
previous paragraph reduces this to checking that $\Delta^{cl}(w^{-1}\cdot (-2\rho))\cong
T_{w^{-1}}^{-1}\Delta^{cl}(-2\rho)$.

To prove the latter isomorphism, recall that, for a simple reflection $s\in W$, the endo-functor $T_s^{-1}$ of $D^b(\Ocat^{[0]})$
is given by the cone of $\Theta_s\rightarrow \mathsf{1}$ (where the source functor is in homological
degree $0$). Note that, for $u\in W$, we have $ u^{-1}s\cdot (-2\rho)>u^{-1}\cdot (-2\rho)$
if and only if $u<su$ in the Bruhat order. If this is the case, then $\Theta_s \Delta(w^{-1}\cdot (-2\rho))$
fits into a short exact sequence
$$0\rightarrow \Delta^{cl}(u^{-1}s\cdot (-2\rho))\rightarrow
\Theta_s \Delta^{cl}(u^{-1}\cdot (-2\rho))\rightarrow \Delta^{cl}(u^{-1}\cdot (-2\rho))\rightarrow 0$$
and hence $T_s^{-1}\Delta^{cl}(u^{-1}\cdot (-2\rho))\xrightarrow{\sim} \Delta^{cl}(u^{-1}s\cdot (-2\rho))$. It follows that $T^{-1}_{w^{-1}}\Delta^{cl}(-2\rho)\cong \Delta^{cl}(w^{-1}\cdot (-2\rho))$.
This finishes the proof.
\end{proof}

\end{document}